\documentclass[a4paper, reqno]{amsart}

\usepackage{cite}
\usepackage{amsmath}
\usepackage{amssymb}
\usepackage{amsthm}   
\usepackage{empheq}   
\usepackage{graphicx}
\usepackage{color}
\usepackage{appendix}
\usepackage{hyperref} 

\newtheorem{theorem}{Theorem}[section]
\newtheorem{lemma}[theorem]{Lemma}
\newtheorem{proposition}[theorem]{Proposition}
\newtheorem{corollary}[theorem]{Corollary}

\theoremstyle{definition}
\newtheorem{definition}[theorem]{Definition}
\newtheorem{example}[theorem]{Example}

\theoremstyle{remark}
\newtheorem{remark}[theorem]{Remark}

\numberwithin{equation}{section}

\begin{document}

\title[Regular subspaces of Dirichlet forms]{Regular subspaces of Dirichlet forms}

\author{Liping Li}
\address{*School of Mathematical Sciences, Fudan University, Shanghai 200433, China.}
\email{lipingli10@fudan.edu.cn}
\thanks{This paper is collected in Festschrift Masatoshi Fukushima,  In Honor of Masatoshi Fukushima's Sanju, pp: 397-420, 2015.}

\author{Jiangang Ying}
\address{School of Mathematical Sciences, Fudan University, Shanghai 200433, China.}
\email{jgying@fudan.edu.cn}





\begin{abstract}
The regular subspaces of a Dirichlet form are the regular Dirichlet forms that inherit the original form but possess smaller domains. The two problems we are concerned are: (1) the existence of regular subspaces of a fixed Dirichlet form, (2) the characterization of the regular subspaces if exists. In this paper, we will first research the structure of regular subspaces for a fixed Dirichlet form. The main results indicate that the jumping and killing measures of each  regular subspace are just equal to that of the original Dirichlet form. By using the independent coupling of Dirichlet forms and some celebrated probabilistic transformations, we will study the existence and characterization of the regular subspaces of local Dirichlet forms.
\end{abstract}

\maketitle
\section{Introduction}\label{ra_sec1}
Let $E$ be a locally compact {separable metric space}, $m$ a  Radon measure fully supported on $E$, then $L^2(E,m)$ is a Hilbert space. A non-negative definite symmetric bilinear form $\mathcal{E}$ with domain $\mathcal{F}$ densely defining on $L^2(E,m)$ is called a Dirichlet form and denoted by $(\mathcal{E},\mathcal{F})$ if it is closed and Markovian. $(T_t)_{t>0}, (G_\alpha)_{\alpha>0}$ are its semigroup and resolvent. Define:
 \[
	 \mathcal{E}_\alpha(u,v):= \mathcal{E}(u,v)+\alpha\cdot (u,v)_m,\quad {u,v\in \mathcal{F}}, \alpha\geq 0.
 \]
We then denote  the space of all real continuous functions on $E$ by $C(E)$ and its subspace of continuous functions with compact support (resp. bounded continuous functions, {one order continuous differentiable functions with compact supports, the continuous functions which converge to zero at infinite}) by $C_\mathrm{c}(E)$ (resp. $C_\mathrm{b}(E),C^1_\mathrm{c}(E), C_0(E)$). A Dirichlet form $(\mathcal{E},\mathcal{F})$ is called regular if $\mathcal{F}\cap C_\mathrm{c}(E)$ is dense in $\mathcal{F}$ with {$\mathcal{E}_1^{\frac{1}{2}}$-norm} and dense in $C_\mathrm{c}(E)$ with uniform norm. A core of $\mathcal{E}$ is by definition a subset $\mathcal{C}$ of $\mathcal{F}\cap C_\mathrm{c}(E)$ such that $\mathcal{C}$ is dense in $\mathcal{F}$ with {$\mathcal{E}_1^{\frac{1}{2}}$-norm} and dense in $C_\mathrm{c}(E)$ with uniform norm.  We refer the definitions of standard core and special standard core to \S1.1 of  \cite{FU}. {The 1-capacity of regular Dirichlet form $(\mathcal{E},\mathcal{F})$ is always denoted by  \textbf{Cap}, see(2.1.2) and(2.1.3) of  \cite{FU}. An increasing sequence $\{F_k, k\geq 1\}$ of closed sets of $E$ is an \emph{$\mathcal{E}$-nest} if $\cup_{k\geq 1}\mathcal{F}_{F_k}$ is $\mathcal{E}_1^{\frac{1}{2}}$-dense in $\mathcal{F}$ where $\mathcal{F}_{F_k}:= \{f\in \mathcal{F}: f=0\; m\text{-a.e. on }F_k\}$. A subset $N$ of $E$ is \emph{$\mathcal{E}$-polar} if there is an $\mathcal{E}$-nest $\{F_k,k\geq 1\}$ such that $N\subset \cap_{k\geq 1}(E\setminus F_k)$. A statement depending on $x\in A$ is said to hold $\mathcal{E}$-quasi-everywhere ($\mathcal{E}$-q.e. in abbreviation) on $A$ if there is an $\mathcal{E}$-polar set $N\subset A$ such that the statement is true for every $x\in A\setminus N$. A function $f$ on $E$ is said to be $\mathcal{E}$-quasi-continuous if there is an $\mathcal{E}$-nest $\{F_k,k\geq 1\}$ such that $f|_{F_k}$ is finite and continuous on $F_k$ for each $k\geq 1$, which will be denoted in abbreviation as $f\in C(\{F_k\})$. In fact, a set $N$ is $\mathcal{E}$-polar if and only if $N$ is of zero 1-capacity, i.e. $\text{Cap}(N)=0$. An increasing sequence of closed sets $\{F_k\}$ is an $\mathcal{E}$-nest if and only if $\lim_{k\rightarrow \infty}\text{Cap}(K\setminus F_k)=0$ for any compact set $K\subset E$. If there is no confusion with the fixed regular Dirichlet form $(\mathcal{E},\mathcal{F})$, for convenience we drop ``$\mathcal{E}$-" from the terminology, then ``$\mathcal{E}$-polar", an ``$\mathcal{E}$-nest" and ``$\mathcal{E}$-quasi-continuous" will simply be  called \emph{polar}, a \emph{nest} and \emph{quasi continuous} respectively. We call an increasing sequence $\{F_k\}$ of closed sets a \emph{Cap-nest} if $\lim_{k\rightarrow \infty}\text{Cap}(E\setminus F_k)=0$. Any Cap-nest is a ($\mathcal{E}$-)nest but not vice versa. However, a function $f$ is quasi continuous if and only if there is a Cap-nest $\{F_k\}$ such that $f\in C(\{F_k\})$, i.e. $f|_{F_k}$ is finite and continuous on $F_k$ for each $k\geq 1$. By \text{Theorem 2.1.7} of  \cite{FU}, any function $u\in \mathcal{F}$ admits a quasi continuous version $\tilde{u}$, i.e. $\tilde{u}$ is quasi continuous and $\tilde{u}=u$ $m$-a.e.}

For an $m$-measurable function $u$, {the support $\text{\text{supp}}[|u|\cdot m]$ of the measure $|u|\cdot m$ is simply denoted by $\text{\text{supp}}[u]$}. We say that $(\mathcal{E},\mathcal{F})$ possesses the local property if $\mathcal{E}(u,v)=0$ for any $u,v\in\mathcal{F}$ with disjoint compact supports. Moreover $(\mathcal{E},\mathcal{F})$ is called {strongly local} if $ \mathcal{E}(u,v)=0$ for any $u,v\in\mathcal{F}$ with compact supports and $v$ is constant on a neighbourhood of $\text{\text{supp}}[u]$. For more details, see   
 \cite{FU}.


The basic concept of this paper is the regular subspace of a Dirichlet form.

\begin{definition}\label{D1}
    Let $(\mathcal{E},\mathcal{F}),(\mathcal{E}',\mathcal{F}')$ be two regular Dirichlet forms on $L^2(E,m)$, we say $(\mathcal{E}',\mathcal{F}')$ is a {\emph{regular subspace}} of $(\mathcal{E},\mathcal{F})$ if
        \begin{equation}\label{DEF1}
            \mathcal{F'} \subset \mathcal{F}, \qquad \mathcal{E}(u,v)=\mathcal{E'}(u,v) \quad  {u,v \in \mathcal{F'}}.
        \end{equation}
We denote it by $(\mathcal{E',F'})\prec (\mathcal{E,F})$ or just $\mathcal{E'} \prec \mathcal{E}$ for simple. If in addition $\mathcal{F}'\neq \mathcal{F}$, $(\mathcal{E',F'})$ is said to be a proper subspace of $(\mathcal{E},\mathcal{F})$.

If $Y$ and $Z$ are the corresponding Hunt processes of $(\mathcal{E},\mathcal{F})$ and $(\mathcal{E}',\mathcal{F}')$, we say $Z$ is {an $R$-\emph{subprocess}} of $Y$ and denote it by $Z\prec Y$.
\end{definition}

{\begin{remark}\label{RMCC}
	Let $(\mathcal{E',F'})\prec (\mathcal{E,F})$ and $\text{Cap}$, $\text{Cap}'$ the 1-capacity of $(\mathcal{E,F})$, $(\mathcal{E',F'})$ respectively. Clearly 
	\[
		\text{Cap}'(A)\geq \text{Cap}(A),
	\]
	for any $A\subset E$ by the definition of 1-capacity. Hence any $\mathcal{E}'$-polar set is also $\mathcal{E}$-polar, any $\mathcal{E}'$-(resp. $\text{Cap}'$-)nest is also $\mathcal{E}$-(resp. $\text{Cap}$-)nest and any $\mathcal{E}'$-quasi-continuous function is also $\mathcal{E}$-quasi-continuous. 
\end{remark}}




Two problems that we are concerned in this paper are the existence of regular subspaces and the characterization of regular subspaces if they exist for a fixed regular Dirichlet form. These problems were raised firstly in     \cite{FMG} which showed that all regular subspaces of one-dimensional Brownian motion can be achieved by a time change of additive functional firstly and a spatial transform secondly.
It is a pity that they hold only for one-dimensional local cases, in other words, the diffusions on $\mathbf{R}$, due to the method used in that paper,  see    \cite{XPJ}.

The structure of this paper is as follows. We will first research the structure of regular subspaces for a fixed Dirichlet form. In \text{\S\ref{HJK}}, we will prove that  the jumping and killing measures of every regular subspace are just equal to that of  the original Dirichlet  form, see \text{Theorem \ref{THM3}}. In \text{\S\ref{IUST}}, we will discuss several probabilistic transformations such as  time changes, killing and resurrected transformations. It is shown that if $(\tilde{\mathcal{E}},\tilde{\mathcal{F}})$ is a regular Dirichlet form which is transformed from another one, say $(\mathcal{E}, \mathcal{F})$, then every regular subspace of $(\tilde{\mathcal{E}},\tilde{\mathcal{F}})$  can be produced by the same transformation on some regular subspace of $(\mathcal{E}, \mathcal{F})$. If $(\mathcal{E}, \mathcal{F})$ and $(\tilde{\mathcal{E}},\tilde{\mathcal{F}})$ have the above relationship, we say that \textbf{they are two Dirichlet forms with the same structure of regular subspaces}. These transformations are powerful tools to deal with what we are concerned. For example, the killing part of a Dirichlet form doesn't play a role in generating a regular subspace, we could always assume that it has no killing inside (by a resurrection if necessary), see \text{\S\ref{REs}}. On the other hand, by spatial homeomorphic transformation and time changes, we can easily extend the results of one-dimensional Brownian motion to one-dimensional diffusions. This is the main topic of     \cite{XPJ}, see \text{Example \ref{EXA1}}.

In \text{\S\ref{EXIS}}, we will study the regular subspaces of a local Dirichlet form. By introducing the independent coupling of Dirichlet forms, we will prove that the associated Dirichlet forms of multidimensional Brownian motions always have proper regular subspaces. On the other hand, the corresponding R-subprocesses of multidimensional Brownian motions may be acted as the independent products of R-subprocess of one-dimensional Brownian motion, see \text{Proposition \ref{PROP6}} and \text{Theorem \ref{PROP7}}. In \text{\S\ref{PRAB}}, we will extend the similar results to the planar reflecting and absorbing Brownian motion on any domain of $\mathbf{R}^2$ by using the transformation method introduced in  \text{\S\ref{IUST}}. It indicates that the  associated Dirichlet forms of absorbing or reflecting Brownian motions on any planar domains always possess the same structure of regular subspaces. In addition, the associated Dirichlet forms of the reflecting and absorbing Brownian motions on arbitrary domains of $\mathbf{R}^d$ always possess proper regular subspaces not only for the two dimensional cases, see \S\ref{ARBAD}. At last, we will also prove that the uniform-ellipticity-type {strongly local} Dirichlet forms on arbitrary domains always possess proper regular subspaces, see \text{Proposition \ref{UED}}. 

\section{Structure of regular subspaces}\label{CBT}
In this section, we will research the structure of regular subspaces. We first conclude that the jumping and killing measures of every regular subspace are equal to that of the original Dirichlet form. Then we will discuss several probabilistic transformations which could become powerful tools to deal with what we concerned about some specific Dirichlet forms.

\subsection{Jumping and killing measures of regular subspaces}\label{HJK}

The following theorem gives a basic property about regular subspaces that the jumping and killing measures of regular subspace must be the same as that of the original Dirichlet form. This is the essential reason for \text{Theorem 4.1(1)} of  \cite{XPJ} which is just the diffusion cases.

Let us first review the Beurling-Deny formulae\footnote{See \text{Theorem 3.2.1} of  \cite{FU}.}  which state that any regular Dirichlet form $(\mathcal{E},\mathcal{F})$ on $L^2(E,m)$ can be expressed for any $u,v\in \mathcal{F}$ as follows:
    \begin{equation}\label{BDC}
        \begin{split}
                \mathcal{E}(u,v)=& \mathcal{E}^{(\mathrm{c})}(u,v)+ {\int_{E\times E \backslash d}(\tilde{u}(x)-\tilde{u}(y))(\tilde{v}(x)-\tilde{v}(y))J(dx,dy) }\\
                &+\int_E \tilde{u}(x)\tilde{v}(x)k(dx)
        \end{split}
    \end{equation}
where $\mathcal{E}^{(\mathrm{c})}$ is a symmetric form, satisfying the {strongly local} property and Markovian property\footnote{See \S1.4 of   \cite{FU}.}, $J$ is a symmetric positive Radon measure on the product space $E\times E$ off the diagonal $d$ called the {\emph{jumping measure}} and $k$ is a positive Radon measure on $E$ called the {\emph{killing measure}}. They both charge no set of zero capacity, i.e. if $A$ is measurable and $\text{Cap}(A)=0$, then $k(A)=0, J(A\times E)=0$. Note that the following equation is satisfied:
    \begin{equation}\label{DJ1}
        \mathcal{E}(u,v)={-2\int_{E\times E}u(x)v(y)J(dx,dy)}
    \end{equation}
for any $u,v\in \mathcal{F} \cap C_\mathrm{c}(E)$ with disjoint supports. The triple $(\mathcal{E}^{(\mathrm{c})},J,k)$ is uniquely determined by $\mathcal{E}$ and we call it the Beurling-Deny triple.

{We refer the definitions of the energy measure $\mu_{<u>}$ as well as its strongly local part measure $\mu_{<u>}^c$ of $u\in \mathcal{F}_b$ to \S3.2 of  \cite{FU}. Moreover, for $u,v\in \mathcal{F}_b$ put 
\begin{equation}\label{MUC}
\mu^c_{<u,v>}:=\frac{1}{2}(\mu^c_{<u+v>}-\mu^c_{<u>}-\mu^c_{<v>})
\end{equation}
which is a bounded signed measure. 
We then have
\begin{equation}\label{IEFD}
	\int_E f d\mu^c_{<u,v>}=\mathcal{E}^{(c)}(uf,v)+\mathcal{E}^{(c)}(vf,u)-\mathcal{E}^{(c)}(uv,f)
\end{equation}}
for any $f\in \mathcal{F}\cap C_c(E)$.

 \begin{theorem}\label{THM3}
  {  Let $(\mathcal{E},\mathcal{F}),(\mathcal{E}',\mathcal{F}')$ be two regular Dirichlet forms on $L^2(E,m)$ with Beurling-Deny triples $(\mathcal{E}^{(\mathrm{c})},J,k)$ and $(\mathcal{E}^{(\mathrm{c})'},J',k')$ respectively. Then $(\mathcal{E}',\mathcal{F}')\prec (\mathcal{E}, \mathcal{F})$ if and only if  $\mathcal{F}'\cap C_\mathrm{c} \subset \mathcal{F}\cap C_\mathrm{c}, J=J',k=k'$ and $\mu^c_{<u,v>}={\mu'}^{c}_{<u,v>}$ for any $u,v\in \mathcal{F}'\cap C_\mathrm{c}$ where ${\mu'}^{c}_{<u,v>}$ is defined similarly as(\ref{MUC}) with respect to $(\mathcal{E',F'})$.}
\end{theorem}
\begin{proof}
We first prove the ``if'' part. Note that the quasi-continuous versions of an appropriate function for $\mathcal{E}$ and $\mathcal{E}'$ are different, but for any $u,v\in \mathcal{F}'\cap C_\mathrm{c}$, it is easy to verify that $\mathcal{E}(u,v)= \mathcal{E}'(u,v)$ under these assumptions. We only need to prove $\mathcal{F}'\subset \mathcal{F}$ and $\mathcal{E}(u,u)=\mathcal{E}'(u,u)$ for any $ u\in \mathcal{F}'$ (then we have $\mathcal{E}(u,v)=\mathcal{E}'(u,v)$ for any $ u,v\in \mathcal{F}'$ by using polarization formula).

For any $u\in \mathcal{F}'$, as $\mathcal{E}'$ is regular there exists a sequence $\{u_n\}\subset \mathcal{F}'\cap C_\mathrm{c}$ such that $u_n\rightarrow u$ in {${\mathcal{E}'}_1^{\frac{1}{2}}$-norm}. Because $\mathcal{E}_1(u_n,u_n)=\mathcal{E}'_1(u_n,u_n)$ is convergent to $\mathcal{E}'_1(u,u)$, $u_n$ is $\mathcal{E}_1$-Cauchy. Since $\mathcal{E}$ is closed, there exists $v\in \mathcal{F}$ such that $u_n\rightarrow v$ in {$\mathcal{E}_1^{\frac{1}{2}}$-norm}, especially $u_n\rightarrow v$ in $L^2(E,m)$. But $u_n\rightarrow u$ in $L^2(E,m)$ and hence $u=v\in \mathcal{F}$ and $\mathcal{E}(u,u)=\mathcal{E}'(u,u)$.

{For the other direction, clearly $\mathcal{F}'\cap C_c(E)\subset \mathcal{F}\cap C_c(E)$. We only need to prove $J=J',k=k'$. In fact, if we have $J=J',k=k'$, then for any $u,v\in \mathcal{F}'\cap C_c(E)$, it holds that
\[
	\mathcal{E}^{(c)}(u,v)={\mathcal{E}}^{(c)'}(u,v)
\]
because of (\ref{BDC}). Then we can deduces from (\ref{IEFD}) that
\[
	\int_E fd\mu^c_{<u,v>}=\int_Efd{\mu'}^c_{<u,v>}
\]
for any $f\in \mathcal{F}'\cap C_c(E)$. Hence $\mu^c_{<u,v>}={\mu'}^c_{<u,v>}$ by using the regularity of $(\mathcal{E',F'})$.}

To prove $J=J'$, let $G_1,G_2$ be two disjoint relatively compact open sets. For any $u,v\in \mathcal{F}'\cap C_\mathrm{c}$ with $\text{\text{supp}}[u]\subset G_1, \text{\text{supp}}[v]\subset G_2$, we know from (\ref{DEF1}) and (\ref{DJ1})  that
    \begin{eqnarray*}
    {-2\int u(x)v(y)J(dx,dy)}&=&\mathcal{E}(u,v)\\
                            &=&\mathcal{E}'(u,v)={-2\int u(x)v(y)J'(dx,dy)}.
    \end{eqnarray*}
As $\mathcal{E}'$ is regular, $\mathcal{F}'\cap C_\mathrm{c}$ is dense in $C_\mathrm{c}$. Then one can easily check that
\[
	J(G_1\times G_2)=J'(G_1\times G_2).
\]
On the other hand, for any two disjoint compact sets $K_1,K_2$, one can find two sequences of relatively compact open sets $\{G^1_n\},\{G^2_n\}$ such that $G^1_n \cap G^2_n = \varnothing$ for any $n$ and $G^1_n\downarrow K_1,\; G^2_n\downarrow K_2$. It follows from $J(G^1_n\times G^2_n)=J'(G^1_n\times G^2_n)$ that $J(K_1\times K_2)=J(K_1\times K_2)$. We conclude that, for any two disjoint Borel sets $B_1,B_2$ we have $J(B_1\times B_2)=J'(B_1\times B_2)$.  Then a simple argument gains directly that $J=J'$.

We now prove $k=k'$.

For any relatively compact open set $G$, there exists $u\in \mathcal{F}'\cap C_\mathrm{c}$ such that $u|_G \equiv 1$. In fact, we can choose a function $u_0\in C_\mathrm{c}(E)$, such that $u_0|_{\bar{G}} \geqslant 2$. Since $\mathcal{F}\cap C_\mathrm{c}(E)$ is dense in $C_\mathrm{c}(E)$ with uniform norm, for $\varepsilon$ small enough there exists $u_{0,\varepsilon}\in \mathcal{F}\cap C_\mathrm{c}(E)$ such that $\| u_0-u_{0,\varepsilon}\| <\varepsilon$. Let $\varphi_\varepsilon(t):=t-(\varepsilon \wedge t)\vee (-\varepsilon)$, and $\psi(t):=(0\vee t)\wedge 1$, which are normal contractions. It is easy to check that $u=\psi(\varphi_\varepsilon(u_{0,\varepsilon}))\in \mathcal{F} \cap C_\mathrm{c}(E)$ and equals to 1 on $\bar{G}$. Let $v$ be any function in $\mathcal{F}'\cap C_\mathrm{c}$ with $\text{\text{supp}}[v]\subset G$. Because $\mathcal{E}^{(\mathrm{c})}$ and ${\mathcal{E}^{(\mathrm{c})}}'$ are {strongly local}, $\mathcal{E}^{(\mathrm{c})}(u,v)={\mathcal{E}^{(\mathrm{c})}}'(u,v)=0$. As we have proved $J=J'$, it follows from (\ref{DEF1}) and (\ref{BDC}) that
    \[
        \int u(x)v(x)k(dx)=\int u(x)v(x)k'(dx),
    \]
for any $ v\in \mathcal{F}'\cap C_\mathrm{c}$ with $\text{supp}[v]\subset G,\;u|_G\equiv 1$. Then we have $\int v(x)k(dx)=\int v(x)k'(dx)$. By using the regularity of $\mathcal{E}'$, it is easy to get $k=k'$ which completes the proof.
\end{proof}


The following corollaries are interesting and not difficult to verify.

\begin{corollary}\label{COR1}
    Let $(\mathcal{E},\mathcal{F})$, $(\mathcal{E}',\mathcal{F}')$ be two regular Dirichlet forms on $L^2(E,m)$ and $(\mathcal{E}',\mathcal{F}')\prec (\mathcal{E}, \mathcal{F})$.
   \begin{enumerate}
      \item If $(\mathcal{E},\mathcal{F})$ is local, so is $(\mathcal{E}',\mathcal{F}')$. Moreover if $(\mathcal{E},\mathcal{F})$ is {strongly local}, so is $(\mathcal{E}',\mathcal{F}')$ too.
      \item If $(\mathcal{E},\mathcal{F})$ is a pure-jumping type Dirichlet form which means $\mathcal{E}^{(\mathrm{c})}=0,\,k=0$, then $(\mathcal{E}',\mathcal{F}')$ is pure-jumping type too.
      \item If $(\mathcal{E},\mathcal{F})$ is a pure-killing Dirichlet form which means $\mathcal{E}^{(\mathrm{c})}=0,\,J=0$, then it must be $\mathcal{F}=L^2(E,m+k)$ and $(\mathcal{E}',\mathcal{F}')=(\mathcal{E},\mathcal{F})$.
      \item If the jumping part or killing part of $(\mathcal{E},\mathcal{F})$ does not disappear, it also does not disappear for that of $(\mathcal{E}',\mathcal{F}')$.
   \end{enumerate}
\end{corollary}

\begin{proof}
    Clearly {(1)(2) and (4)  hold} by \text{Theorem \ref{THM3}}.

    (3) When $\mathcal{E}^{(\mathrm{c})}=0$ and $J=0$ it is obvious that $\mathcal{E}(u,u)=\int u^2dk$ and {$\|u\|_{\mathcal{E}_1} := \mathcal{E}(u,u)^{\frac{1}{2}}= \| u\|_{L^2(m+k)}$} for any $u\in \mathcal{F}$. By the regularity of $\mathcal{E}$,  $\mathcal{F}\cap C_\mathrm{c}$ is dense in $C_\mathrm{c}$ with uniform norm, then it is also dense in $L^2(E,m+k)$ with norm $\| \cdot \|_{L^2(m+k)}=\|\cdot \|_{\mathcal{E}_1}$. For any $u\in L^2(E,m+k)$, there exists a sequence $u_n$ in $\mathcal{F}\cap C_\mathrm{c}$ such that $\|u_n-u\|_{\mathcal{E}_1}\rightarrow 0$. Hence $u\in \mathcal{F}$ and $\mathcal{F}=L^2(E,m+k)$. It is also easy to check that $\mathcal{F}'=L^2(E,m+k)$ which means $(\mathcal{E}',\mathcal{F}')=(\mathcal{E},\mathcal{F})$. That completes the proof.
\end{proof}

Before showing another corollary, we should give some necessary statements about the L\'evy-type Dirichlet forms. A L\'{e}vy-type Dirichlet form on $L^2(\mathbf{R}^d,dx)$ is generated by a symmetric convolution semigroup $\{\nu_t,t>0\}$ on $\mathbf{R}^d$. Its corresponding Markov process is just the so-called symmetric L\'{e}vy process. The celebrated L\'{e}vy-Khinchin formula under the symmetric assumption reads as follows:
\begin{equation}\label{5}
        \hat{\nu}_t(x)=\exp\{-t\psi(x)\}
\end{equation}
where
\begin{equation}
\psi(x)=\frac{1}{2}(Sx,x)+ \int_{\mathbf{R}^d}(1-\text{cos}(x,y))j(dy),
    \end{equation}
$S$ is a non-negative definite $d\times d$ symmetric matrix and $j$ is a symmetric Radon measure on $\mathbf{R}^d\setminus \{0\}$ such that $\int_{\mathbf{R}^d\setminus \{0\}} 1\wedge |x|^2 j(dx) <\infty$. The corresponding regular Dirichlet form is
\begin{equation}\label{Levy-Kinchin}
        \begin{split}
        \mathcal{D}[\mathcal{E}]&=\{u\in L^2(\mathbf{R}^d): \int_{\mathbf{R}^d} |\hat{u}|(x)^2\psi(x)dx<\infty\}, \\
               \mathcal{E}(u,v)&=\int_{\mathbf{R}^d} \hat{u}(x){\bar{\hat{v}}(x)}\psi(x)dx \quad \forall u,v\in \mathcal{D}[\mathcal{E}],
       \end{split}
\end{equation}
where $\hat{u}$ means the Fourier transform of $u$. Note that the symmetric convolution semigroup, as well as the regular Dirichlet form, are characterized by the pair $(S,j)$.

The {strongly local} part of $(\mathcal{E},\mathcal{D}[\mathcal{E}])$ which is determinate by $S$ is:
\begin{equation}\label{ESQ}
    \mathcal{E}^\mathrm{S}(u,u)=\frac{1}{2}\int(S\nabla u,\nabla u)dx
\end{equation}
and the jumping part can be rewritten as
\begin{equation}\label{EJQ}
    \mathcal{E}^j(u,u)=\frac{1}{2}\int_{\mathbf{R}^d\times (\mathbf{R}^d\setminus \{0\})}(u(x+y)-u(x))^2j(dy)dx.
\end{equation}
Of course $\mathcal{E}=\mathcal{E}^\mathrm{S}+\mathcal{E}^j$. On the other hand, as $S$ is a non-negative definite symmetric matrix there exists an orthogonal matrix $P$, i.e. $P'=P^{-1}$, such that
   { \begin{equation}\label{matrix}
     P'SP=\left(
               \begin{array}{cccccc}
                 \lambda_1 & 0 & \cdots & \cdots & \cdots & 0 \\
                 0 & \cdots & 0 & \cdots & \cdots & 0 \\
                 0 & 0 & \lambda_r & 0 & \cdots & 0 \\
                 0 & \cdots & \cdots & 0 & \cdots & 0 \\
                 0 & \cdots & \cdots & \cdots & \cdots & 0 \\
                 0 & \cdots & \cdots & \cdots & \cdots & 0 \\
               \end{array}
             \right)_{d\times d}=\text{diag}\{\lambda_1,\cdots, \lambda_r, 0,\cdots, 0\}.
    \end{equation}}
It is just the diagonalization of $S$ with the rank $r$ and $\lambda_1, \cdots, \lambda_r$ are its positive eigenvalues, i.e. $\lambda_i>0$ for any $1\leqslant i \leqslant r$. Note that $S\neq 0$ means $r\geq 1$. For any $u\in \mathcal{D}[\mathcal{E}]$, \text{supp}ose $u(x)= \varphi(Px)$ for a unique function $\varphi$. Then
    \[\begin{split}
     \mathcal{E}^\mathrm{S}(u,u)=&\frac{1}{2}\int(S\nabla \varphi(Px),\nabla \varphi(Px))dx \\
                            =&\frac{1}{2}\int(P'SP\nabla_y \varphi(y),\nabla_y \varphi(y))dP'y \\
                            =&\frac{1}{2}\int(P'SP\nabla_y \varphi(y),\nabla_y \varphi(y))dy,
    \end{split}\]
hence,
    \begin{equation}\label{ES}
        \mathcal{E}^\mathrm{S}(u,u)= \frac{1}{2}\sum_{i=1}^r \lambda_i \int(\frac{\partial{\varphi}}{\partial{y_i}})^2dy.
    \end{equation}

\begin{corollary}
 If $(\mathcal{E},\mathcal{F})$ is a non-pure-jumping L\'{e}vy-type Dirichlet form, i.e. the scatter coefficient $S\neq 0$ and $(\mathcal{E}',\mathcal{F}')\prec (\mathcal{E},\mathcal{F})$. Then $(\mathcal{E}',\mathcal{F}')$ mustn't be pure-jumping type, in other words, the {strongly local} part of $(\mathcal{E}',\mathcal{F}')$ never disappears.
\end{corollary}
\begin{proof} Suppose that $(\mathcal{E},\mathcal{F})$ is just the L\'{e}vy Dirichlet form (\ref{Levy-Kinchin}) such that $S\neq 0$, i.e. $r\geq 1$ and $(\mathcal{E}',\mathcal{F}')\prec (\mathcal{E},\mathcal{F})$. If $(\mathcal{E}',\mathcal{F}')$ is a pure-jumping Dirichlet form, then $\mathcal{E}^\mathrm{S}(u,u)=0$ for any $u\in \mathcal{F}'$. For any $u(x)=\varphi(Px) \in \mathcal{F}'\cap C_\mathrm{c}$, it follows from (\ref{ES}) that
    \[
        \int (\frac{\partial \varphi}{\partial y_1})^2dy=0
    \]
and hence $\frac{\partial \varphi}{\partial y_1}=0$, a.e. Note that $\frac{\partial \varphi}{\partial y_1}$ is the weak derivative of $\varphi$.

Set $y=(y_1,\bar{y})\in \mathbf{R}^d$, i.e. $\bar{y}\in {\mathbf{R}^{d-1}}$ and $\varphi_{\bar{y}}(y_1):= \varphi(y_1,\bar{y})$. Take any $\psi_1(y_1)\in {C_\mathrm{c}^{\infty}(\mathbf{R})}, \psi_2(\bar{y})\in {C_\mathrm{c}^{\infty}(\mathbf{R}^{d-1})}$, and write $\psi(y):=\psi_1(y_1)\psi_2(\bar{y})\in C_\mathrm{c}^{\infty}(\mathbf{R}^d)$. By the definition of weak derivative,
    \[
        0=(\frac{\partial \varphi}{\partial y_1},\psi)=-(\varphi,\psi_1'\psi_2)=-\int \psi_2(\bar{y})d\bar{y}\int \varphi_{\bar{y}}(y_1)\psi_1'(y_1)dy_1.
    \]
Because $\int \varphi_{\bar{y}}(y_1)\psi_1'(y_1)dy_1$ is continuous with respect to $\bar{y}$ and $\psi_2(\bar{y})$ is arbitrary in ${C_\mathrm{c}^{\infty}(\mathbf{R}^{d-1})}$, we have
    \[
        \int \varphi_{\bar{y}}(y_1)\psi_1'(y_1)dy_1=0,\text{ for any fix }\bar{y}\in {\mathbf{R}^{d-1}}.
    \]
Then $0=\int \varphi_{\bar{y}}(y_1)\psi_1'(y_1)dy_1=-(\varphi_{\bar{y}}'(y_1),\psi_1(y_1))$. It follows from \text{Lemma 3.31} and \text{Corollary 3.32} of     \cite{SS}  that $\varphi_{\bar{y}}(y_1)$ is independent of $y_1$. But $\varphi_{\bar{y}}$ is continuous, hence {$\varphi_{\bar{y}}(y_1)\equiv \text{const.}$ for each $\bar{y}$} and we have
    \[
        \varphi(y_1,\bar{y})=\bar{\varphi}(\bar{y}),
    \]
for some {$\bar{\varphi}\in C_\mathrm{c}(\mathbf{R}^{d-1})$, which implies $\varphi \equiv 0$ because of $\varphi\in C_c(\mathbf{R}^d)$}. This contradicts that $\mathcal{F}'\cap C_\mathrm{c}(\mathbf{R}^d)$ is dense in $C_\mathrm{c}(\mathbf{R}^d)$ with uniform norm. We then conclude that $(\mathcal{E}',\mathcal{F}')$ mustn't be pure-jumping type.
\end{proof}

\subsection{Probabilistic transformations}\label{IUST}

An $m$-measurable set $A\subset E$ is said to be $(T_t)$-invariant if $T_t(1_Af)=1_A\cdot T_tf$, $m$-a.e for any $f\in L^2(E)$ and $t>0$. A Dirichlet form with semigroup $(T_t)_{t>0}$ is called {\emph{$m$-irreducible}} if every $(T_t)$-invariant set $A$ is $m$-trivial, i.e. $m(A)=0$, or $m(E\backslash A)=0.$ In this section we will discuss several celebrated probabilistic transformations.


\subsubsection{Spatial homeomorphous transformation}

Let $(\mathcal{E},\mathcal{F})$ be a regular Dirichlet form on $L^2(E,m)$. $\hat{E}$ is another topological space and there exists a homeomorphism mapping $j:E\rightarrow \hat{E}$. Then $\hat{E}$ is also a locally compact {separable metric space}. Set $\hat{m}:= m\circ j^{-1}$ which is a Radon measure fully supported on $\hat{E}$. Define
\begin{equation}\label{SHT}
\begin{split}
 \hat{\mathcal{F}}&:= \{\hat{u}\in L^2(\hat{E},\hat{m}): \hat{u}\circ j\in \mathcal{F}    \}, \\
  \hat{\mathcal{E}}(\hat{u},\hat{v})&:= \mathcal{E}(\hat{u}\circ j,\hat{v}\circ j),\quad \forall u,v\in \hat{\mathcal{F}}.
\end{split}\end{equation}
 We denote $(\hat{\mathcal{E}},\hat{\mathcal{F}})$ by $j(\mathcal{E},\mathcal{F})$. Note that $(\mathcal{E},\mathcal{F})=j^{-1}(\hat{\mathcal{E}},\hat{\mathcal{F}})$.

\begin{lemma}\label{LEMMA1}
Let $L^2(E,m)$, $L^2(\hat{E},\hat{m})$, $(\mathcal{E},\mathcal{F})$ and $(\hat{\mathcal{E}},\hat{\mathcal{F}})$ be above. 
\begin{enumerate}
               \item $(\hat{\mathcal{E}},\hat{\mathcal{F}})$ is a regular Dirichlet form on $L^2(\hat{E},\hat{m})$.
               \item $(\hat{\mathcal{E}},\hat{\mathcal{F}})$ is irreducible if and only if $(\mathcal{E},\mathcal{F})$ is irreducible.
               \item $B$ is $\mathcal{E}$-polar if and only if $\hat{B}=j(B)$ is $\hat{\mathcal{E}}$-polar, $u$ is $\mathcal{E}$-quasi-continuous if and only if $\hat{u}=u\circ j^{-1}$ is $\hat{\mathcal{E}}$-quasi-continuous.
             \end{enumerate}
\end{lemma}
\begin{proof}
The proofs of (1)(3) are trivial and we omit them.

For (2) we need only to prove that if $A$ is $(\mathcal{E},\mathcal{F})$-invariant, $\hat{A}=j(A)$ is $(\hat{\mathcal{E}},\hat{\mathcal{F}})$-invariant. This can be checked by \text{Propsiotion 2.1.6} of   \cite{CF} and for any $ \hat{u}\in \hat{\mathcal{F}}$,
    \[
        1_{\hat{A}}\cdot \hat{u}=1_A\circ j^{-1}\cdot u\circ j^{-1}=(1_A\cdot u)\circ j^{-1},
    \]
    where $u=\hat{u}\circ j\in \mathcal{F}$.
\end{proof}

\begin{proposition}\label{PROP2}
   { $(\mathcal{E},\mathcal{F})$ and $(\hat{\mathcal{E}},\hat{\mathcal{F}})$ are two Dirichlet forms with the same structure of regular subspaces with respect to the spatial homeomorphic transformation $j$. Then for any $(\hat{\mathcal{E}}',\hat{\mathcal{F}}')\prec (\hat{\mathcal{E}},\hat{\mathcal{F}})$, there exists a unique  $(\mathcal{E}',\mathcal{F}')\prec (\mathcal{E},\mathcal{F})$ such that $(\hat{\mathcal{E}}',\hat{\mathcal{F}}')=j(\mathcal{E}',\mathcal{F}')$.}
\end{proposition}
\begin{proof}
 Note that $j$ is reversible. So the conclusions are trivial because $(\mathcal{E}',\mathcal{F}')\prec (\mathcal{E},\mathcal{F})$ if and only if $(\hat{\mathcal{E}'},\hat{\mathcal{F}'}):= j(\mathcal{E}',\mathcal{F}')\prec (\hat{\mathcal{E}},\hat{\mathcal{F}})$
\end{proof}

Combining with \text{Lemma \ref{LEMMA1}(2)},  we have

\begin{corollary}\label{COR3}
    Let $(\mathcal{E},\mathcal{F})$ and $(\hat{\mathcal{E}},\hat{\mathcal{F}})$ be above. For any irreducible $(\hat{\mathcal{E}}',\hat{\mathcal{F}}')\prec (\hat{\mathcal{E}},\hat{\mathcal{F}})$, there exists a unique irreducible $(\mathcal{E}',\mathcal{F}')\prec (\mathcal{E},\mathcal{F})$ such that $(\hat{\mathcal{E}}',\hat{\mathcal{F}}')=j(\mathcal{E}',\mathcal{F}')$.
\end{corollary}

\subsubsection{Time changes with full quasi support}

Assume that $(\mathcal{E},\mathcal{F})$ is a regular Dirichlet form on $L^2(E,m)$ and $X$ is its corresponding Hunt process.  Denote the set of Radon smooth measures on $E$ with full quasi support by $S_E$. Let $\mu\in S_E$, $(A_t)_{t\geq 0}$ its corresponding positive continuous additive functional and $(\tau_t)_{t\geq 0}$ the inverse of $(A_t)_{t\geq 0}$. Then the time-changed process $(X_{\tau_t})_{t\geq 0}$ is $\mu$-symmetric and its Dirichlet form $(\check{\mathcal{E}},\check{\mathcal{F}})$ on $L^2(E,\mu)$ can be presented as
\begin{equation}\label{12}
\begin{split}
    \check{\mathcal{F}}&=\mathcal{F}_{\mathrm{e}}\cap L^2(E,\mu),\\
\check{\mathcal{E}}(u,u)&=\mathcal{E}(u,u),\quad { u\in \check{\mathcal{F}}},
\end{split}\end{equation}
where $\mathcal{F}_{\mathrm{e}}$ is the extended space of $(\mathcal{E},\mathcal{F})$. Moreover $(\check{\mathcal{E}},\check{\mathcal{F}})$ is regular as $\mu$ is Radon and its extended space $\check{\mathcal{F}}_\mathrm{e}=\mathcal{F}_{\mathrm{e}}$. Note that the quasi-notions (nests, polar sets and quasi-continuous functions) of $(\mathcal{E},\mathcal{F})$ and $(\check{\mathcal{E}},\check{\mathcal{F}})$ are equivalent. We denote $(\check{\mathcal{E}},\check{\mathcal{F}})$ by $\mu(\mathcal{E},\mathcal{F})$ if $\mu\in S_E$. The above procedures are reversible because $m$ is a Radon smooth measure with quasi support $E$ with respect to $(\check{\mathcal{E}},\check{\mathcal{F}})$. In other words, $(\mathcal{E},\mathcal{F})=m(\check{\mathcal{E}},\check{\mathcal{F}})$.

Similarly $\check{S}_E$ (resp. $S'_E$) represents the set of all Radon smooth measures on $E$ with full quasi support with respect to $(\check{\mathcal{E}},\check{\mathcal{F}})$ (resp. $(\mathcal{E}',\mathcal{F}')$ which is another Dirichlet form).

\begin{lemma}\label{LEMMA2}
Let $(\mathcal{E},\mathcal{F})$ and {$(\mathcal{E}',\mathcal{F}')$} be two Dirichlet forms on $L^2(E,m)$.
\begin{enumerate}
\item It holds that $\mathcal{F}'\subset \mathcal{F}$ if and only if $\mathcal{F}'_\mathrm{e}\subset \mathcal{F}_{\mathrm{e}}$. In particular, $\mathcal{F}'\subsetneqq \mathcal{F}$ if and only if $\mathcal{F}'_\mathrm{e}\subsetneqq \mathcal{F}_{\mathrm{e}}$.
               \item  If $(\mathcal{E}',\mathcal{F}')\prec(\mathcal{E},\mathcal{F})$ and $\mu \in S_E$, then $\mu\in S'_E$.
               \item Assume $(\mathcal{E},\mathcal{F})$ is regular on $L^2(E,m)$ and $\mu\in S_E$. Then $(\mathcal{E},\mathcal{F})$ is irreducible if and only if $\mu(\mathcal{E},\mathcal{F})$ is irreducible.
\end{enumerate}
\end{lemma}
\begin{proof}
    (1) The second assertion follows directly from the first one. For the first one, the if part is because $\mathcal{F}=\mathcal{F}_{\mathrm{e}}\cap L^2(E,m)$ (resp. $\mathcal{F}'=\mathcal{F}'_\mathrm{e}\cap L^2(E,m)$). The only if part is obvious by the definition of extended space.

    (2) {By Remark \ref{RMCC}, any $\text{Cap}'$-zero-capacity set is also of Cap-zero-capacity. As  $\mu$ is Radon and smooth with respect to $(\mathcal{E,F})$, $\mu$ is also smooth with respect to $(\mathcal{E}',\mathcal{F}')$.} It suffices to prove that  the quasi support of $\mu$ with respect to $(\mathcal{E}',\mathcal{F}')$ is $E$. As $\mu \in S_E$, it holds that $u=0$ $\mu$-a.e. if and only if $u=0$ $\mathcal{E}$-q.e.  for any $\mathcal{E}$-quasi continuous function $u\in \mathcal{F}$ by \text{Theorem 3.3.5} of  \cite{CF}. For any $\mathcal{E}'$-quasi continuous $v\in \mathcal{F}'$, it also holds that $v\in\mathcal{F}$ and $v$ is $\mathcal{E}$-quasi continuous. Therefore if $v=0$ $\mu$-a.e., then $u=0$ $\mathcal{E}$-q.e, and hence also $m$-a.e. But $v$ is $\mathcal{E}'$-quasi continuous, we have $v=0$ $\mathcal{E}'$-q.e. On the other hand,  it follows from $v=0$ $\mathcal{E}'$-q.e. and $\mu \in S_E$ that $v=0$ $\mu$-a.e. In other words, $v=0$ $\mu$-a.e. if and only if $v=0$ $\mathcal{E}'$-q.e for any $\mathcal{E}'$-quasi continuous function $v\in \mathcal{F}'$. Hence $\mu \in S'_E$ by using \text{Theorem 3.3.5} of  \cite{CF} again.

    (3) It is clear by using (\ref{12}), $\check{\mathcal{F}}_\mathrm{e}=\mathcal{F}_{\mathrm{e}}$ and \text{Proposition 2.1.6} of  \cite{CF}.
\end{proof}

\begin{remark}
It follows from \text{Lemma \ref{LEMMA2}(1)} that  the Dirichlet spaces $\mathcal{F}$ and $\mathcal{F}'$ appeared in \text{Definition \ref{D1}} can be replaced by their extended spaces $\mathcal{F}_{\mathrm{e}}$ and $\mathcal{F}'_\mathrm{e}$ respectively.
\end{remark}

\begin{proposition}\label{PROP3}
  If  $(\mathcal{E},\mathcal{F})$ is a \ regular Dirichlet form on $L^2(E,m)$ and $\mu \in S_E$, then $(\check{\mathcal{E}},\check{\mathcal{F}})=\mu(\mathcal{E},\mathcal{F})$ has the same structure of regular subspaces as $(\mathcal{E},\mathcal{F})$ with respect to $\mu$ in the meaning that for any $(\check{\mathcal{E}}',\check{\mathcal{F}}')\prec (\check{\mathcal{E}},\check{\mathcal{F}})$, there exists a unique  $(\mathcal{E}',\mathcal{F}')\prec (\mathcal{E},\mathcal{F})$ such that $(\check{\mathcal{E}}',\check{\mathcal{F}}')=\mu(\mathcal{E}',\mathcal{F}')$.
\end{proposition}

\begin{proof}
Clearly $m\in \check{S}_E$, hence it follows from \text{Lemma \ref{LEMMA2}(2)} that $m\in \check{S}'_E$. Set  $(\mathcal{E}',\mathcal{F}'):=m(\check{\mathcal{E}}',\check{\mathcal{F}}')$. As
\[
       \mathcal{F}_{\mathrm{e}}= \check{\mathcal{F}}_\mathrm{e},\quad \mathcal{F}'_\mathrm{e}=\check{\mathcal{F}}'_\mathrm{e},
    \]
it follows from \text{Lemma \ref{LEMMA2}(1)} that  $(\mathcal{E}',\mathcal{F}')\prec(\mathcal{E},\mathcal{F})$. By the similar method, it is easy to check that if another $(\mathcal{E}'',\mathcal{F}'')\prec (\mathcal{E},\mathcal{F})$ such that $(\check{\mathcal{E}}',\check{\mathcal{F}}')=\mu(\mathcal{E}'',\mathcal{F}'')$, it must hold that $(\mathcal{E}'',\mathcal{F}'')=m(\check{\mathcal{E}}',\check{\mathcal{F}}')=(\mathcal{E}',\mathcal{F}')$. The uniqueness is proved.
\end{proof}

Combining with \text{Lemma \ref{LEMMA2}(3)}, we have

\begin{corollary}\label{COR4}
    Under the conditions of \text{Proposition \ref{PROP3}}, if $(\check{\mathcal{E}}',\check{\mathcal{F}}')$ is irreducible, then the corresponding $(\mathcal{E}',\mathcal{F}')$ is also irreducible.
\end{corollary}

\subsubsection{Killing and resurrected transformations}\label{KRT}

Let $(\mathcal{E},\mathcal{F})$ be a regular Dirichlet form on $L^2(E,m)$ with no killing inside (or zero killing measure) and $X$ its corresponding Hunt process. Assume that the functions in $\mathcal{F}$ always take their quasi-continuous versions. Take a positive Radon measure $k$ on $E$ charging no $\mathcal{E}$-zero-capacity sets and the perturbed (killing) Dirichlet form $(\mathcal{E}^k,\mathcal{F}^k)$ of $(\mathcal{E},\mathcal{F})$ with respect to $k$ on $L^2(E,m)$ is defined by
\begin{equation}\label{Perturbed}
        \begin{split}
                \mathcal{F}^k&:=\mathcal{F}\cap L^2(E,k), \\
                    \mathcal{E}^k(u,v)&:=\mathcal{E}(u,v)+(u,v)_k,\quad  {u,v\in \mathcal{F}^k}.
              \end{split}
\end{equation}
By \text{Theorem 6.1.2} of   \cite{FU}, $(\mathcal{E}^k,\mathcal{F}^k)$  is a regular Dirichlet form whose corresponding symmetric Markov process is the canonical subprocess $X^k$  of $X$  relative to the multiplicative functional $(e^{-A_t})_{t\geq 0}$   where $A_t$  is the corresponding PCAF of smooth measure $k$. In particular, the killing measure of $(\mathcal{E}^k,\mathcal{F}^k)$ is just the Radon measure $k$. {Note that the perturbed Dirichlet form has the equivalent $\mathcal{E}$-quasi notions as the original one by Theorem 5.1.4 of  \cite{CF}, i.e. $\mathcal{E}$-polar (resp. $\mathcal{E}$-nest, $\mathcal{E}$-quasi-continuous) and $\mathcal{E}^k$-polar (resp. $\mathcal{E}^k$-nest, $\mathcal{E}^k$-quasi-continuous) are equivalent. Moreover, any special standard core of $(\mathcal{E},\mathcal{F})$ remains to be a special standard core of $(\mathcal{E}^k,\mathcal{F}^k)$.}

The opposite procedure is called ``resurrection". Let $(\mathcal{E},\mathcal{F})$ be a regular Dirichlet form  on $L^2(E,m)$ with killing measure $k\neq 0$.   For any $u,v\in \mathcal{F}_{\mathrm{e}}$, define
\begin{equation}\label{14}
    \mathcal{E}^{\mathrm{res}}(u,v):=\mathcal{E}^{(\mathrm{c})}(u,v)+\frac{1}{2}{\int}_{E\times E\setminus d}(u(x)-u(y))(v(x)-v(y))J(dxdy)
\end{equation}
and $m_k:= m+k$ where $(\mathcal{E}^{(\mathrm{c})},J,k)$ is the Beurling-Deny triple of $(\mathcal{E},\mathcal{F})$. Then it follows from
    \[\begin{split}
       \mathcal{F}&=\mathcal{F}_{\mathrm{e}}\cap L^2(E,m_k), \\  \mathcal{E}_1(u,v)&=\mathcal{E}^{\mathrm{res}}(u,v)+(u,v)_{m_k},\quad { u,v\in \mathcal{F}},
    \end{split}\]
that $(\mathcal{E}^{\mathrm{res}},\mathcal{F})$ is a regular Dirichlet form on $L^2(E,m_k)$.  Let $(\mathcal{E}^{\mathrm{res}},\mathcal{F}^{\mathrm{res}}_\mathrm{e})$ be the extended Dirichlet space of $(\mathcal{E}^{\mathrm{res}},\mathcal{F})$. Define
\begin{equation}\label{15}
    \mathcal{F}^{\mathrm{res}}=\mathcal{F}^{\mathrm{res}}_\mathrm{e}\cap L^2(E,m),
\end{equation}
then $(\mathcal{E}^{\mathrm{res}},\mathcal{F}^{\mathrm{res}})$ is the so-called resurrected Dirichlet form of $(\mathcal{E},\mathcal{F})$ which has no killing inside. It is regular on  $L^2(E,m)$ and shares the same {$\mathcal{E}$-quasi notions} with $(\mathcal{E},\mathcal{F})$. Moreover any special standard core of $(\mathcal{E},\mathcal{F})$ is a core of $(\mathcal{E}^{\mathrm{res}},\mathcal{F}^{\mathrm{res}})$. Note that the extended Dirichlet space of $(\mathcal{E}^{\mathrm{res}},\mathcal{F}^{\mathrm{res}})$ on $L^2(E,m)$ is just $(\mathcal{E}^{\mathrm{res}},\mathcal{F}^{\mathrm{res}}_\mathrm{e})$. For more details, see  \citen{FU, CF}.

Killing and resurrected transformations are reciprocal in the following sense: for a regular Dirichlet form with no killing inside, kill it firstly by a Radon smooth measure and then resurrect. The final Dirichlet form is the same as the original one. On the contrary, for a regular Dirichlet form with killing measure $k\neq 0$, resurrect it firstly and then kill it by $k$. The final Dirichlet form is also the same as the original one.

We are now ready to prove that the killing and resurrected transformations are both equivalent relationships between Dirichlet forms, in the meaning that they have the same structures holding regular subspaces between the transformed Dirichlet form  by killing or resurrection and  the original Dirichlet form. The following lemma is trivial and we omit its proof.

\begin{lemma}\label{3}
    If  $(\mathcal{E}',\mathcal{F}')\prec(\mathcal{E},\mathcal{F})$ and $k$ is a Radon smooth measure with respect to $(\mathcal{E},\mathcal{F})$, then $k$ is also a Radon smooth measure with respect to $(\mathcal{E}',\mathcal{F}')$.
\end{lemma}

\begin{lemma}\label{LEMMA4}
The following assertions about the irreducibility hold.
    \begin{enumerate}
    	\item Let $(\mathcal{E},\mathcal{F})$ be a regular Dirichlet form on $L^2(E,m)$ with no killing inside and $k$ a Radon smooth measure on $E$. {Then the perturbed Dirichlet form $(\mathcal{E}^k,\mathcal{F}^k)$ is irreducible if and only if $(\mathcal{E},\mathcal{F})$ is irreducible.}
     	 \item Let $(\mathcal{E},\mathcal{F})$ be a regular Dirichlet form on $L^2(E,m)$ with killing measure $k\neq 0$. {Then $(\mathcal{E},\mathcal{F})$ is irreducible if and only if $(\mathcal{E}^{\mathrm{res}},\mathcal{F}^{\mathrm{res}})$ is irreducible}.
    \end{enumerate}
\end{lemma}
\begin{proof}
    (1) We need only to prove $\mathcal{E}$-invariant sets are equivalent to $\mathcal{E}^k$-invariant sets. In fact, let $A$ be an $\mathcal{E}$-invariant set, then $1_A\cdot u\in \mathcal{F}$ for any $u\in \mathcal{F}^k\subset \mathcal{F}$. Hence $1_A\cdot u\in \mathcal{F}^k$ because $1_A\cdot u\in L^2(E,k)$. By \text{Proposition 2.1.6} of   \cite{CF} and
    \[
        (u,u)_k=(1_A\cdot u,1_A\cdot u)_k+(1_{A^c}\cdot u,1_{A^c}\cdot u)_k,
    \]
    we have $A$ is an $\mathcal{E}^k$-invariant set. {On the other hand, assume $B$ is an $\mathcal{E}^k$-invariant set, then $u\in \mathcal{F}\cap C_c(E)$ implies $u\in \mathcal{F}^k$, hence $1_B\cdot u\in \mathcal{F}^k\subset \mathcal{F}$ and 
    \[
    	\mathcal{E}^k(u,u)=\mathcal{E}^k(1_B\cdot u, 1_B\cdot u)+\mathcal{E}^k(1_{B^c}\cdot u,1_{B^c}\cdot u).
    \]
    This shows that $u\in \mathcal{F}\cap C_c(E)$ implies $1_B\cdot u\in \mathcal{F}$ and 
    \[
    \mathcal{E}(u,u)=\mathcal{E}(1_B\cdot u, 1_B\cdot u)+\mathcal{E}(1_{B^c}\cdot u,1_{B^c}\cdot u)
    \]
    because of $(u,u)_k=(1_B\cdot u,1_B\cdot u)_k+(1_{B^c}\cdot u,1_{B^c}\cdot u)_k$. From this equation and the regularity of $(\mathcal{E,F})$, we can deduce that $u\in \mathcal{F}$ implies $1_B\cdot u\in \mathcal{F}$ and 
    \[
    \mathcal{E}(u,u)=\mathcal{E}(1_B\cdot u, 1_B\cdot u)+\mathcal{E}(1_{B^c}\cdot u,1_{B^c}\cdot u).
    \]}

    (2) They are just the former conclusions in an opposite direction.
\end{proof}

\begin{proposition}
The following conclusions about the structure of the regular subspaces hold.
    \begin{enumerate}
    \item Let $(\mathcal{E},\mathcal{F})$ be a regular Dirichlet form on $L^2(E,m)$ with no killing inside and $k$ a Radon smooth measure on $E$, then the perturbed Dirichlet form $(\mathcal{E}^k,\mathcal{F}^k)$ and $(\mathcal{E},\mathcal{F})$ are two Dirichlet forms with the same structure of regular subspaces with respect to $k$ in the meaning that, for any $(\mathcal{E}^{'k},\mathcal{F}^{'k})\prec (\mathcal{E}^k,\mathcal{F}^k)$, there exists a unique $(\mathcal{E}',\mathcal{F}')\prec (\mathcal{E},\mathcal{F})$ such that $(\mathcal{E}^{'k},\mathcal{F}^{'k})$ is the perturbed Dirichlet form of $(\mathcal{E},\mathcal{F})$ with respect to $k$.
      \item If $(\mathcal{E},\mathcal{F})$ is a regular Dirichlet form on $L^2(E,m)$ with killing measure $k\neq 0$, then $(\mathcal{E}^{\mathrm{res}},\mathcal{F}^\text{{res}})$ and $(\mathcal{E},\mathcal{F})$ are two Dirichlet forms with the same structure of regular subspaces with respect to the resurrection in the meaning that, for any $(\mathcal{E}^{\mathrm{res}'},\mathcal{F}^{\mathrm{res}'})\prec (\mathcal{E}^{\mathrm{res}},\mathcal{F}^{\mathrm{res}})$, there exists a unique $(\mathcal{E}',\mathcal{F}')\prec (\mathcal{E},\mathcal{F})$ such that $(\mathcal{E}^{\mathrm{res}'},\mathcal{F}^{\mathrm{res}'})$  is the resurrected Dirichlet form of $(\mathcal{E}',\mathcal{F}')$.
    \end{enumerate}
\end{proposition}
\begin{proof}
    (1) Take any $(\mathcal{E}',\mathcal{F}')\prec(\mathcal{E},\mathcal{F})$, then  $k$ is also  a  Radon smooth measure with respect to $(\mathcal{E}',\mathcal{F}')$ by \text{Lemma \ref{3}}. It follows from (\ref{Perturbed}) that
    \[
        (\mathcal{E}'^k,\mathcal{F}'^k)\prec(\mathcal{E}^k,\mathcal{F}^k).
    \]
We have known that $(\mathcal{E}^k,\mathcal{F}^k)$ is regular with killing measure $k$ and its resurrected Dirichlet form is just $(\mathcal{E},\mathcal{F})$. It suffices to prove that for any $(\mathcal{E}^{k'},\mathcal{F}^{k'})\prec(\mathcal{E}^k,\mathcal{F}^k)$, the resurrected Dirichlet form of $(\mathcal{E}^{k'},\mathcal{F}^{k'})$ is a regular subspace of $(\mathcal{E},\mathcal{F})$. In fact, $(\mathcal{E}^{k',\mathrm{res}},\mathcal{F}^{k'})\prec(\mathcal{E}^{k,\mathrm{res}},\mathcal{F}^k)$ on $L^2(E,m_k)$, where $\mathcal{E}^{k,\mathrm{res}}$ and $\mathcal{E}^{k',\mathrm{res}}$ are the resurrected forms of $\mathcal{E}^k$ and $\mathcal{E}^{k'}$ defined by (\ref{14}) respectively. By \text{Lemma \ref{LEMMA2}(1)} and (\ref{15})
    \[
        \mathcal{F}^{k',\mathrm{res}}_\mathrm{e}\subset \mathcal{F}^{k,\mathrm{res}}_\mathrm{e},\quad \mathcal{F}^{k',\mathrm{res}}\subset \mathcal{F}^{k,\mathrm{res}}=\mathcal{F}.
    \]
  Hence
    \[
        (\mathcal{E}^{k',\mathrm{res}},\mathcal{F}^{k',\mathrm{res}})\prec(\mathcal{E},\mathcal{F}).
    \]

    (2) We need only to change the order of the two steps in the first assertion.
\end{proof}


\subsubsection{Conclusions and examples}\label{REs}

From the above discussions, we could summarize that:

\begin{theorem}\label{SPT}
	Let $(\mathcal{E},\mathcal{F})$ be a regular Dirichlet form on $L^2(E,m)$ and $T$ a transformation described above, i.e. $T$ is one of the transformations including spatial homeomorphic transformation, time change with full quasi support, killing and resurrection. Define $(\tilde{\mathcal{E}}, \tilde{\mathcal{F}}):=T(\mathcal{E},\mathcal{F})$,  which is the transformed Dirichlet form of $(\mathcal{E},\mathcal{F})$ by $T$. Then $(\tilde{\mathcal{E}}, \tilde{\mathcal{F}})$ and $(\mathcal{E},\mathcal{F})$ have the same structure of regular subspaces with respect to the transformation $T$ in the meaning that for any $(\tilde{\mathcal{E}}', \tilde{\mathcal{F}}')\prec (\tilde{\mathcal{E}}, \tilde{\mathcal{F}})$, there exists a unique $(\mathcal{E}',\mathcal{F}')\prec  (\mathcal{E},\mathcal{F})$ such that $(\tilde{\mathcal{E}}', \tilde{\mathcal{F}}')=T(\mathcal{E}',\mathcal{F}')$. Conversely, it holds that $(\mathcal{E},\mathcal{F})=T^{-1}(\tilde{\mathcal{E}}, \tilde{\mathcal{F}})$ and they also have the same structure of regular subspaces with respect to the inverse transformation $T^{-1}$ in the similar sense.
\end{theorem}

In \text{Corollary \ref{COR1}(3)}, we have proved that if the Dirichlet form $(\mathcal{E},\mathcal{F})$ is pure-killing type, i.e. $\mathcal{E}(u,u)=(u,u)_k$ for some Radon measure $k$, then it has no proper regular subspaces. In addition,  for general Dirichlet forms,  the killing part does not play a role in generating a proper regular subspaces because it follows from \text{\S\ref{KRT}} that the resurrected Dirichlet form is equivalent to the original one up to the resurrection. In other words, adding or erasing the killing part of a Dirichlet form does not really change the structure of its regular subspaces. Hence it is not necessary to consider the killing part for the discussions about the regular subspaces of a Dirichlet form. We assume all the Dirichlet forms appeared in the sequel have no killing inside unless otherwise stated.

It is discussed in     \cite{FMG} that how to characterize the regular subspaces of one-dimensional (reflecting) Brownian motion. By \text{Theorem \ref{THM3}} every corresponding Markov process of its regular subspace must be a diffusion without killing inside. If the diffusion is irreducible, it can be characterized by a scale function $s$ which is uniquely determined up to a linear transformation.  The additional condition on $s$ is that $s'=0$ or $1$ a.e.  Using the above theorem we can give a simple characterization to the regular subspaces of one-dimensional diffusions.

\begin{example}\label{EXA1}
    Let $I\subset \mathbf{R}$ be an interval and $(\mathcal{E},\mathcal{F})$ an irreducible {strongly local} Dirichlet form on $L^2(I,m)$ whose corresponding Markov process is an irreducible diffusion $(X_t)_{t\geq 0}$. Write $I=\langle a,b\rangle$ with two endpoints, $-\infty\leq a<b\leq \infty$, which may or may not be in $I$ ($X$ is reflecting or absorbing at the endpoints respectively). Clearly $m$ is the speed measure of $X$ (unique up to a constant) and let $s$ be  its scale function  which is a strictly increasing and continuous function on $I$. Hence
    \[
        \mathcal{E}(u,u)=\frac{1}{2}\int_I(\frac{du}{ds})^2ds,\quad { u\in \mathcal{F}}.
    \]

    Set $J:=s(I)\subset \mathbf{R}$, then $s^{-1}:J\rightarrow I$ is a homeomorphism mapping. Let $(B_t)_{t\geq 0}$ be the Brownian motion on $J$ which is reflecting or absorbing at the finite endpoint of $J$ if $X$ is absorbing or reflecting at the corresponding endpoint of $I$. Then $(X_t)_{t\geq 0}$ is the time-changed Markov process of $(s^{-1}(B_t))_{t\geq 0}$ by $m$, i.e. a homeomorphism transformation $s^{-1}$ firstly and a time change by $m$ next on $(B_t)_{t\geq 0}$. Applying \text{Theorem \ref{SPT}}, we know that the associated Dirichlet forms of the one-dimensional irreducible diffusion $X$ on $I$ and the corresponding Brownian motion $B$ on $J$ have the same structure of regular subspaces with respect to the transformation $s^{-1}$ and the time change by Revuz measure $m$.
In particular, by the characterization of the regular subspaces of one-dimensional Brownian motion\footnote{See \text{Theorem \ref{THMBM}}.}, any irreducible regular subspace $(\mathcal{E}',\mathcal{F}')$ of $(\mathcal{E},\mathcal{F})$ can be characterized by the speed measure $m$ and another scale function $s'$ on $I$, i.e. a strictly increasing and continuous function on $I$, such that
    \[
        ds'\ll ds\text{ and } \frac{ds'}{ds}=0\text{ or } 1,\quad ds\text{-a.e.}
    \]
Note that this is the main results of     \cite{XPJ}.
\end{example}

More applications of \text{Theorem \ref{SPT}} will be found in the following text, especially in dealing with the planar reflecting or absorbing Brownian motions in \text{\S\ref{PRAB}}.

\section{Regular subspaces of local Dirichlet forms}\label{EXIS}
The authors  proved in     \cite{FMG} that one-dimensional Brownian motion always possesses proper regular subspaces which are characterized by a special scale function. We shall study the similar problems for the symmetric diffusions on multidimensional space. Before that let's reviews the results of one-dimensional Brownian motion.

\subsection{One-dimensional Brownian motion}\label{BMC}
It is known that $(\frac{1}{2}\mathbf{D},H^1(\mathbf{R}))$ is the corresponding Dirichlet form of one-dimensional Brownian motion $(B_t)_{t\geq0}$ where
\begin{equation}\label{ODBM}
 \begin{split}
        H^1(\mathbf{R}) &= \{u\in L^2(\mathbf{R}): u\text{ is absolutely continuous, }u'\in L^2(\mathbf{R}) \}, \\
        \mathbf{D}(u,v)&=\int u'(x)v'(x)dx.
        \end{split}
    \end{equation}
An irreducible diffusion on $\mathbf{R}$ associated with a local Dirichlet form can be characterized by a scale function $s$, a speed measure $m$ and a killing measure $k$ where $s$ is a strictly increasing and continuous function on $\mathbf{R}$ and $m,k$ are Radon measures on {$\mathbf{R}$}. We may and will assume $k=0$ because of the notes after \text{Theorem \ref{SPT}}. Define
\begin{equation}\label{SR}
    \begin{split}
    \mathbf{S}(\mathbf{R}):= \{s:\mathbf{R}\rightarrow \mathbf{R}\,|\, &s \text{ is strictly increasing and continuous,}\;s(0)=0\;\\
                                &\text{and} \;s'(x)=0\text{ or }1\text{ a.e. } \},
    \end{split}
\end{equation}
\begin{equation}
    \hat{\mathbf{S}}(\mathbf{R}):= \{s\in \mathbf{S}(\mathbf{R}):\;s(\pm \infty)=\pm \infty\}.
\end{equation}
For any $s\in \mathbf{S}(\mathbf{R})$, define
\begin{equation}\label{SBM}
 \begin{split}
        \mathcal{F}^{(s)}&:= \{u\in L^2(\mathbf{R}):\;u\ll s,\;\int_{\mathbf{R}}(\frac{du}{ds})^2ds<\infty\} \\
        \mathcal{E}^{(s)}(u,v)&:= \frac{1}{2}\int_{\mathbf{R}} \frac{du}{ds}\frac{dv}{ds}ds,\quad u,v\in \mathcal{F}^{(s)},
        \end{split}
    \end{equation}
   {where $u\ll s$ means $u$ is absolutely continuous with respect to $s$, in other words, there exists an absolutely continuous function $\phi$ such that $u=\phi\circ s$.}
Although $s^{-1}$ may not be absolutely continuous, $ds^{-1}$ is a smooth measure of $(B_t,P^x)$ because the only polar set of $(B_t)$ is null set. Let $(A_t)$ be the PCAF associated with  $ds^{-1}$ and $(\tau_t)$  the inverse of $(A_t)$. Then we have:

\begin{theorem}[ \cite{FMG}, 2005]\label{THMBM}
The following assertions about the regular subspaces of one-dimensional Brownian motion hold.
\begin{enumerate}
\item The Dirichlet form $(\mathcal{F}^{(s)},\mathcal{E}^{(s)})$ is {strongly local} and regular on $L^2(\mathbf{R})$ whose associated diffusion has the scale function $s$ up to a linear transform, speed measure $dx$. Hence the generator of $(\mathcal{F}^{(s)},\mathcal{E}^{(s)})$ is
 \[
 	A^{(s)}=\frac{1}{2}\frac{d}{dx}\frac{d}{ds}.
 	\]
  \item The class $C_\mathrm{c}^1(s):=\{{\varphi\circ s}:\varphi \in C_\mathrm{c}^1(J)\text{ where } J=s(\mathbf{R})\}$ is a special standard core of $(\mathcal{F}^{(s)},\mathcal{E}^{(s)})$.
  \item For any $s\in \mathbf{S}(\mathbf{R})$, $(\mathcal{F}^{(s)},\mathcal{E}^{(s)})\prec (\frac{1}{2}\mathbf{D},H^1(\mathbf{R}))$. Moreover $\mathcal{F}^{(s)}\subsetneqq H^1(\mathbf{R})$ if and only if $|E_s|>0$ where $E_s=\{x\in \mathbf{R}: s'(x)=0\}$ and $|E_s|$ means the Lebesgue measure of $E_s$.
  \item The regular subspace $(\mathcal{E}',\mathcal{F}')$ of $(\frac{1}{2}\mathbf{D},H^1(\mathbf{R}))$ is recurrent if and only if $(\mathcal{E}',\mathcal{F}')=(\mathcal{F}^{(s)},\mathcal{E}^{(s)})$ for some $s\in \hat{\mathbf{S}}(\mathbf{R})$. Moreover its associated diffusion is $(X^s_t,P^{s(x)})$ where $X^s_t=s^{-1}(B_{\tau_t})$.
\end{enumerate}
\end{theorem}

Note that when $s^{-1}(x)=c(x)+x$, where $c(x)$ is the standard {Cantor function}, it holds that $|E_s|>0$. Moreover, if the Brownian motion lies only on an proper interval $I$ of $\mathbf{R}$, set $I=\langle a,b\rangle$ where two endpoints $a,b$ may or may not be in $I$, accordingly the Brownian motion is reflecting or absorbing at the endpoints.  Let $\mathbf{S}(I)$  be defined as the similar way as (\ref{SR}), replacing $\mathbf{R}$ by $I$ and $s(0)=0$ by $s(a)=0$ if $a>-\infty$.  And the domain of $\mathcal{E}^{(s)}$ should be restricted to a subspace of $\mathcal{F}^{(s)}$ satisfying that the value at the absorbing endpoint of every function in this subspace is zero. The results about the regular subspaces of the reflecting or absorbing Brownian motion on $I$  are similar as \text{Theorem \ref{THMBM}}, see    \cite{FMG}.

We are going to give a brief explanation of $\mathbf{S}(\mathbf{R})$ (or $\mathbf{S}(I)$). Let $(\mathcal{E}',\mathcal{F}')$ be a regular subspace of $(\frac{1}{2}\mathbf{D},H^1(\mathbf{R}))$ whose associated Markov process is denoted by $X$. By \text{Corollary \ref{COR1}}, $X$ is a diffusion on $\mathbf{R}$ and hence it  can be characterized by a strictly increasing and continuous scale function $s$ on $\mathbf{R}$ and a speed measure which is just the Lebesgue measure. Moreover it holds that
\[
  \mathcal{E}'(u,u)= \frac{1}{2}\int (\frac{du}{ds})^2ds,\quad { u\in \mathcal{F}'.}
\]
However $\mathcal{E}'$ is equal to $\frac{1}{2}\mathbf{D}$ on $\mathcal{F}'$ and
\[
\begin{split}
    \frac{1}{2}\mathbf{D}(u,u)=\frac{1}{2}\int (u'(x))^2dx =\frac{1}{2}\int (\frac{du}{ds}\cdot \frac{ds}{dx})^2dx =\frac{1}{2}\int (\frac{du}{ds})^2(s'(x))^2dx,
\end{split}\]
therefore we must assume that $s'(x)^2=s'(x)$ a.e., in other words, $s'(x)=0 \text{ or }1$ a.e, to ensure that
\[
	\mathcal{E}'(u,u)=\frac{1}{2}\mathbf{D}(u,u), \quad \forall u\in \mathcal{F}'.
\]
This is the essential reason for the definition of $\mathbf{S}(\mathbf{R})$.

\subsection{Multidimensional Brownian motions}\label{DBM}

In this section, we will study the regular subspaces of multidimensional Brownian motions. We first introduce the independent coupling of Dirichlet forms.

Let $X$ be {an} $m^X$-symmetric Markov process on a locally compact separable {metric} space $E^X$ whose corresponding Dirichlet form on $L^2(E^X,m^X)$ is $(\mathcal{E}^X,\mathcal{F}^X)$. Similarly, the Markov process $Y$ is $m^Y$-symmetric on another space $E^Y$ whose Dirichlet form is $(\mathcal{E}^Y,\mathcal{F}^Y)$ on $L^2(E^Y,m^Y)$. Note that $m^X$ and $m^Y$ are Radon measures on $E^X$ and $E^Y$ respectively. Our target is the Dirichlet form of the independent coupling process $Z=(X,Y)$ on {$E:=E^X\times E^Y$ with $m:=m^X\times m^Y$.} Before that we {need to give} some notes. Let $(\Omega,\mathcal{M},\mu)$ be a measurable space and $H$ {a real Hilbert space} with the inner product $(\cdot ,\cdot)_H$ and the norm $\|\cdot\|_H$. For any $p\geq 1$,
\[
    L^p(\Omega,H):= \{u:\Omega\rightarrow H|\; \|u(\cdot)\|_H\in L^p(\Omega,\mu)  \}
\]
is a Banach space with the norm $(\int_{\Omega} \|u(\omega)\|_H^p\mu(d\omega))^{\frac{1}{p}}$. Write:
\begin{equation}\label{LXY}
\begin{split}
	L^2(E^Y,\mathcal{F}^X)=\{u&\in L^2(E,m):u(\cdot, y)\in \mathcal{F}^X\;m^Y\text{-a.e. }y,\\&  ||u(\cdot, y)||_{\mathcal{E}_1^X}\in L^2(E^Y, m^Y)\},
\end{split}\end{equation}	

\begin{equation}
\begin{split}
L^2(E^X,\mathcal{F}^Y)=\{u&\in L^2(E,m):u(x,\cdot)\in \mathcal{F}^Y\;m^X\text{-a.e. }x,\\&  ||u(x,\cdot)||_{\mathcal{E}_1^Y}\in L^2(E^X, m^X)\}.
\end{split}
\end{equation}
We have:

\begin{proposition}\label{PROP6}
    Let $X,Y$, $L^2(E^X,m^X)$, $L^2(E^Y,m^Y)$ and $(\mathcal{E}^X,\mathcal{F}^X)$, $(\mathcal{E}^Y,\mathcal{F}^Y)$ be above and $Z=(X,Y)$ the independent coupling process of $X$ and $Y$. Set $E=E^X\times E^Y$ and $m=m^X\times m^Y$. Then $Z$ is $m$-symmetric on $E$ and
    \begin{equation}\label{IC}
     \left\{\begin{array}{ll}
    &\mathcal{F}= L^2(E^Y, \mathcal{F}^X)\cap L^2(E^X,\mathcal{F}^Y),  \\
        &\begin{aligned}\mathcal{E}(u,v)=&\int_Y \mathcal{E}^X(u(\cdot,y),v(\cdot,y))m^Y(dy)\\&+\int_X \mathcal{E}^Y(u(x,\cdot),v(x,\cdot))m^X(dx) ,\quad  {u,v\in \mathcal{F},}\end{aligned}
        \end{array}
        \right.
    \end{equation}
    is a regular Dirichlet form on $L^2(E,m)$ whose corresponding Markov process is $Z$. Moreover if $\mathcal{C}^X$ (resp. $\mathcal{C}^Y$) is a core of $(\mathcal{E}^X,\mathcal{F}^X)$ (resp. $(\mathcal{E}^Y,\mathcal{F}^Y)$), then the tensor product
    \[
        \mathcal{C}:= \mathcal{C}^X\otimes \mathcal{C}^Y
    \]
    is a core of $(\mathcal{E},\mathcal{F})$.
\end{proposition}
\begin{proof}
    Let $(P^X_t),(P^Y_t),(P^Z_t) $ be the transition functions of $X,Y$ and $Z$ respectively. For any $f_1,f_2\in bL^2(E^X,m^X)$ and $g_1,g_2\in bL^2(E^Y,m^Y)$,
    \[\begin{split}
        (P^Z_t(f_1\otimes g_1),f_2\otimes g_2)_m &=(P_t^Xf_1\otimes P_t^Yg_1,f_2\otimes g_2)m \\
                                            &=(P_t^Xf_1,f_2)_{m^X}\cdot (P_t^Yg_1,g_2)_{m^Y} \\
                                            &=(f_1,P_t^Xf_2)_{m^X}\cdot (g_1,P_t^Yg_2)_{m^Y} \\
                                            &=(f_1\otimes g_1,P^Z_t(f_2\otimes g_2))_m.
    \end{split}\]
Therefore $(P^Z_t)_{t\geq 0}$ (i.e. $Z$) is $m$-symmetric.

We are going to prove that $(\mathcal{E},\mathcal{F})$ is a Dirichlet form on $L^2(E,m)$. Clearly $\mathcal{F}\subset L^2(E,m)$ and $\mathcal{C}\subset \mathcal{F}$, hence $\mathcal{F}$ is dense in $L^2(E,m)$. The bilinearity and Markov property are trivial. We need only to prove that $(\mathcal{E},\mathcal{F})$ is closed. For any $\mathcal{E}_2$-Cauchy sequence $\{u_n\}\subset \mathcal{F}$ and any $n,m$,
\[
    \|u_n-u_m\|_{\mathcal{E}_2}^2=\int \|(u_n-u_m)(\cdot,y)\|^2_{\mathcal{E}_1^X}m^Y(dy)+\int \|(u_n-u_m)(x,\cdot)\|^2_{\mathcal{E}_1^Y}m^X(dx),
\]
hence $u_n$ is also $L^2(E^Y,\mathcal{F}^X)$-Cauchy and $L^2(E^X,\mathcal{F}^Y)$-Cauchy. Then there exist $u\in L^2(E^Y,\mathcal{F}^X)$ and $v\in L^2(E^X,\mathcal{F}^Y)$ such that $u_n$ is convergent to $u$ (resp. $v$) in $L^2(E^Y,\mathcal{F}^X)$ (resp. $L^2(E^X,\mathcal{F}^Y)$). But $u_n$ is $L^2(E,m)$-convergent to $u$ (resp. $v$), hence $u=v\in \mathcal{F}$ and $u_n$ is convergent to $u$ in $\mathcal{E}_2$.

We shall now give some notations  for convenience. The semigroups, resolvents and generators of $X,Y,Z$ are denoted by $(T^X_t, G^X_\alpha, A^X)$, $(T^Y_t, G^Y_\alpha, A^Y)$ and  $(T^Z_t, G^Z_\alpha, A^Z)$ respectively. We have proved that $(\mathcal{E},\mathcal{F})$ is a Dirichlet form, whose semigroup, resolvent and generator are denoted by $T_t,G_\alpha$ and $A$. As $\mathcal{C}\subset \mathcal{F}$ is dense in $C_\mathrm{c}(E)$ by Stone-Weierstrass theorem, its closure $\overline{\mathcal{C}}^{\mathcal{E}_1}\subset \mathcal{F}$ is a regular Dirichlet space. The semigroup, resolvent and generator of $(\overline{\mathcal{C}}^{\mathcal{E}_1},\mathcal{E})$ are denoted by $T^{\mathcal{C}}_t,G_\alpha^{\mathcal{C}}$ and $A^{\mathcal{C}}$. Note that for any $ f\in L^2(E^X,m^X),g\in L^2(E^Y,m^Y),t>0$,
\begin{equation}\label{24}
    T^Z_t(f\otimes g)=T^X_tf\otimes T^Y_tg.
\end{equation}
In addition, one can easily check that $\mathcal{F}^X\otimes \mathcal{F}^Y\subset \mathcal{F}^{\mathcal{C}}$.

We are now going to prove that for any $\alpha>0,\;G_\alpha=G_\alpha^Z=G_\alpha^{\mathcal{C}}$ which implies that $(\mathcal{E},\mathcal{F})$ is the associated Dirichlet form of $Z$ with the core $\mathcal{C}$. In particular, $(\mathcal{E},\mathcal{F})$ is regular.

Firstly we claim that $\mathcal{D}(A^X)\otimes \mathcal{D}(A^Y)\subset \mathcal{D}(A^Z)\cap \mathcal{D}(A^{\mathcal{C}})$ and for any $ f\in \mathcal{D}(A^X)$, $g\in\mathcal{D}(A^Y)$, it holds that
\begin{equation}\label{25}
    A^Z(f\otimes g)=A^{\mathcal{C}}(f\otimes g)=A^Xf\otimes g+f\otimes A^Yg.
\end{equation}
In fact, $f\otimes g\in \mathcal{F}^X\otimes \mathcal{F}^Y\subset \mathcal{F}^{\mathcal{C}}$ and for any $h\in \mathcal{F}^{\mathcal{C}}$
\[\begin{split}
    \mathcal{E}(f\otimes g,h)&=\int\mathcal{E}^X(f,h(\cdot,y))g(y)m^Y(dy)+\int\mathcal{E}^Y(g,h(x,\cdot))f(x)m^X(dx) \\
                            &=\int(-A^Xf,h(\cdot,y))_{m^X}g(y)m^Y(dy)   \\
                            &\qquad+\int(-A^Yg,h(x,\cdot))_{m^Y}f(x)m^X(dx)    \\
                            &{=(-A^Xf\cdot g-f\cdot A^Yg,h)_m.}
\end{split}\]
It is similar to see that $f\otimes g\in \mathcal{F}$ and
\[
    {\mathcal{E}(f\otimes g,\tilde{h})=(-A^Xf\cdot g-f\cdot A^Yg,\tilde{h})_m,\quad  \tilde{h}\in \mathcal{F}.}
\]
Hence (\ref{25}) is proved.

On the other hand, $f\in \mathcal{D}(A^X),g\in\mathcal{D}(A^Y)$ implies that $T_t^Xf\in \mathcal{D}(A^X),T^Y_tg\in\mathcal{D}(A^Y)$. By (\ref{24}) and (\ref{25}), $T^Z_t(f\otimes g)\in \mathcal{D}(A^X)\otimes \mathcal{D}(A^Y)$ and
\[\begin{split}
    A^\mathcal{C}T^Z_t(f\otimes g)&=A^\mathcal{C}(T^X_tf\otimes T^Y_tg) \\
                                &{=A^XT^X_t f\cdot T^Y_tg+T^X_tf\cdot A^YT^Y_tg} \\
                                &{=\frac{d}{dt}T^X_tf\cdot T^Y_tg+T^X_tf\cdot\frac{d}{dt}T^Y_tg}\\
                                &=\frac{d}{dt}T^Z_t(f\otimes g)\\
                                &=A^ZT^Z_t(f\otimes g)
\end{split}\]
As $A^\mathcal{C}$ is a closed linear operator and $\int_0^\infty A^ZT^Z_t(f\otimes g)e^{-\alpha t}dt$ is Bochner integrable,
\[\begin{split}
    A^\mathcal{C}G^Z_\alpha(f\otimes g)&= A^\mathcal{C}\int_0^\infty e^{-\alpha t}T^Z_t(f\otimes g)dt \\
                                    &=\int_0^\infty e^{-\alpha t}A^\mathcal{C}T^Z_t(f\otimes g)dt \\
                                    &=\int_0^\infty e^{-\alpha t}A^ZT^Z_t(f\otimes g)dt   \\
                                    &=A^Z\int_0^\infty e^{-\alpha t}T^Z_t(f\otimes g)dt    \\
                                    &=A^ZG^Z_\alpha(f\otimes g).
\end{split}\]
Then it follows that
\[
	(\alpha-A^\mathcal{C})G^Z_\alpha (f\otimes g)=(\alpha-A^Z) G^Z_\alpha (f\otimes g)=f\otimes g.
\]
Making use of $G^\mathcal{C}_\alpha=(\alpha-A^\mathcal{C})^{-1}$, we have
\[
    G^Z_\alpha (f\otimes g)=G^\mathcal{C}_\alpha (f\otimes g)
\]
Similarly we can prove $G^Z_\alpha (f\otimes g)=G_\alpha (f\otimes g)$. Therefore for any $\alpha>0$,
    \[
        G_\alpha=G^\mathcal{C}_\alpha=G^Z_\alpha
    \]
on $\mathcal{D}(A^X)\otimes \mathcal{D}(A^Y)$ which is dense in $L^2(E,m)$. As they are bounded operators on $L^2(E,m)$, it also holds
\[
	G_\alpha h=G^\mathcal{C}_\alpha h=G^Z_\alpha h,{\quad h\in L^2(E,m).}
\]
That completes the proof.
\end{proof}

This proposition can be easily extended to the higher dimensional independent coupling of Markov processes. Let $X^i$ be {an} $m_i$-symmetric Markov process on $E^i$ whose associated Dirichlet form is $(\mathcal{E}^i,\mathcal{F}^i)$ where $1\leq i\leq d$ and $d$ is {an arbitrary natural number.} For any $1\leq i\leq d$, set
\begin{eqnarray*}
E:&=& E^1\times \cdots \times E^d, \\
\hat{E}^i:&=& E^1\times \cdots \times E^{i-1}\times  E^{i+1}\cdots\times  E^d,
\end{eqnarray*}
and for any $x=(x_1,\cdots,x_d)\in E$, $\hat{x}_i\in \hat{E}^i$ is the projection of $x$ on $\hat{E}^i$. For any function $f$ on $E$, set $f_{\hat{x}_i}(x_i):= f(x)$ for any $x\in E$. Similarly $m:= m_1\times \cdots \times m_d$ and $\hat{m}_i$ is its projection on $\hat{E}^i$. Note that if $X^1,\cdots,X^d$ are independent, $X$ is symmetric if and only if {$X^i$ is $m_i$-symmetric} for any $1\leq i\leq d$ and $X$ is Markov if and only if $X^i$ is Markov for any $1\leq i\leq d$. The following corollary is direct by \text{Proposition \ref{PROP6}} and we omit its proof.

\begin{corollary}\label{COR5}
    Let $X=(X^1,\cdots,X^d)$ be the independent coupling process of $X^1,\cdots, X^d$. Then $X$ is $m$-symmetric on $E$ and its associated Dirichlet form is
    \begin{equation}\label{16}
    \begin{split}
\mathcal{F}&=\cap_{i=1}^d L^2(\hat{E}^i, \mathcal{F}^i),  \\
\mathcal{E}(u,v)&=\sum_{i=1}^d\int \mathcal{E}^i(f_{\hat{x}_i}(x_i),g_{\hat{x}_i}(x_i))\hat{m}^i(d\hat{x}_i)\quad { f,g\in \mathcal{F}.}
        \end{split}
    \end{equation}
Note that $L^2(\hat{E}^i, \mathcal{F}^i)$ is defined similar as (\ref{LXY}).    Moreover, $(\mathcal{E},\mathcal{F})$ is regular and if $\mathcal{C}^i$ is a core of $(\mathcal{E}^i,\mathcal{F}^i)$, $1\leq i\leq d$, then the tensor product
    \[
        \mathcal{C}:= \mathcal{C}^1\otimes \cdots \otimes \mathcal{C}^d
    \]
    is a core of $(\mathcal{E},\mathcal{F})$.
\end{corollary}

Since the $d$-dimensional Brownian motion is an independent coupling process of one-dimensional Brownian motions, we could conclude some assertions about the regular subspaces of $d$-mensional Brownian motions. It is well known that $(\frac{1}{2}\mathbf{D},H^1(\mathbf{R}^d))$ (or equivalently the L\'evy-type Dirichlet form with $S=I_d,j=0$) is the associated Dirichlet form of $d$-dimensional Brownian motion. Choose  $s_i\in \mathbf{S}(\mathbf{R})$ for any $1\leq i\leq d$ and let $X^{s_i}$ be the R-subprocess of one-dimensional Brownian motion with respect to $s_i$ appeared in \text{Theorem \ref{THMBM}}. Denote the independent coupling process of $X^{s_1},\cdots,X^{s_d}$ by $(X^{s_1},\cdots,X^{s_d})$ and its corresponding Dirichlet form by $(\mathcal{E}^{s_1,\cdots,s_d},\mathcal{F}^{s_1,\cdots, s_d})$. Define
\[
	\mathcal{C}^{s_1,\cdots,s_d}:=C_\mathrm{c}^1(s_1)\otimes \cdots \otimes C_\mathrm{c}^1(s_d),
\]
where $C_\mathrm{c}^1(s_i)=\{{\phi\circ s_i}:\phi\in C_\mathrm{c}^1(J_i)\text{ where }J_i=s_i(\mathbf{R})\}$ for any $1\leq i\leq d$.



\begin{theorem}\label{PROP7}
   For any $s_1,\cdots,s_d\in \mathbf{S}(\mathbf{R})$,  $X:=(X^{s_1},\cdots,X^{s_d})$ is {an} R-subprocess of $d$-dimensional Brownian motion, in other words, the associated Dirichlet form $(\mathcal{E}^{s_1,\cdots,s_d},\mathcal{F}^{s_1,\cdots, s_d})$ of $X$ is a regular subspace of $(\frac{1}{2}\mathbf{D},H^1(\mathbf{R}^d))$ and $\mathcal{C}^{s_1,\cdots,s_d}$ is a core of $(\mathcal{E}^{s_1,\cdots,s_d},\mathcal{F}^{s_1,\cdots, s_d})$. In particular, if $|E_{s_i}|>0$ for some $1\leq i\leq d$, then the regular subspace $(\mathcal{E}^{s_1,\cdots,s_d},\mathcal{F}^{s_1,\cdots, s_d})$  is a proper one.
\end{theorem}

The proof is obvious by \text{Theorem \ref{THMBM}} and \text{Corollary \ref{COR5}}. Note that $|E_{s_i}|>0$ means $s_i$ is not the scale function of one-dimensional Brownian motion. Usually the R-subprocess of $d$-dimensional Brownian motion is not an independent coupling, but the following corollary is not difficult to prove. So we omit its proof.

\begin{corollary}
    If $X=(X^1,\cdots,X^d)$ is {an} R-subprocess of $d$-dimensional Brownian motion such that $X^1,\cdots,X^d$ are independent and $X^i$ is irreducible for $1\leq i\leq d$, then there exist $s_1,\cdots, s_d\in \mathbf{S}(\mathbf{R})$ such that $X^i=X^{s_i}$, for any $1\leq i\leq d$.
\end{corollary}

\subsection{Planar absorbing and reflecting Brownian motions}\label{PRAB}

Theorem~\ref{PROP7} can be extended to the absorbing or reflecting Brownian motion on a special domain of $\mathbf{R}^d$, such as the rectangle $G=I_1\times \cdots\times I_d\subset \mathbf{R}^d$, where $I_i$ is an interval of $\mathbf{R}$. It is because that on these rectangles the Brownian motions are  independent couplings of one dimensional absorbing or reflecting Browinan motions. However on other usual domains the independent coupling method is invalid.

When $d=2$, it is known that planar Brownian motion is conformal invariant. In fact the absorbing (or reflecting) Brownian motions on some domains are also conformal invariant. Let $\varphi$ be a conformal mapping from domain $U\subset \mathbf{R}^2$ to domain $V\subset \mathbf{R}^2$ and  $B^U$ the absorbing Brownian motion on $U$. Then up to a time change, $\varphi(B^U)$ is an absorbing Brownian motion on $V$\footnote{See \text{Theorem 5.3.1} of  \cite{CF}.}. In addition, by the Riemann mapping theorem, for any non-trivial simply connected open subsets $U,V$  of $\mathbf{R}^2$, there exists a biholomorphic (bijective and holomorphic) mapping $f$ from $U$ to $V$. Obviously $f$ is homeomorphic and conformal mapping. Thus by {Proposition \ref{PROP2} and \ref{PROP3}} we have

\begin{proposition}\label{PROP8}
    Let $U,V$ be two arbitrary non-trivial (non-empty and not $\mathbf{R}^2$) simply connected domains of $\mathbf{R}^2$, then the associated Dirichlet forms of absorbing Brownian motions on $U$ and $V$ are with the same structure of regular subspaces with respect to the corresponding conformal transformation $f$ and time change. In other words, their regular subspaces also have a similar correspondence with respect to these transformations.
\end{proposition}

In particular, it is known that the absorbing Brownian motion on the upper half plane is the independent coupling process of a one-dimensional Brownian motion on $\mathbf{R}$ and a Brownian motion on $(0,\infty)$ which is absorbing at $0$. By Theorem \ref{THMBM}, {Proposition \ref{PROP6} and \ref{PROP8},} we have

\begin{corollary}
    For any non-trivial simply connected domain $U$ of $\mathbf{R}^2$, the associated Dirichlet form of absorbing Brownian motion on $U$ always has proper regular subspaces.
\end{corollary}

The results about reflecting Brownian motion are similar, but the domains should be the Jordan domains in the meaning that they are simply connected and their boundaries are non-self-intersecting continuous loops in the plane. When $U,V$ are Jordan domains the biholomorphic mapping $f:U\rightarrow V$ can be extended to a topological homeomorphism from $\bar{U}$ to $\bar{V}$. If $B^{\bar{U}}$ is the reflecting Brownian motion on $\bar{U}$, then up to a time change, $f(B^{\bar{U}})$ is a reflecting Brownian motion on $\bar{V}$\footnote{See \S5.3($2^\circ$) of  \cite{CF}.}. Hence we have

\begin{proposition}
    If $U,V$ are two Jordan domains of $\mathbf{R}^2$, then the associated Dirichlet forms of reflecting Brownian motions on $\bar{U}$ and $\bar{V}$ are with the same structure of regular subspaces with respect to the corresponding biholomorphic mapping and time change. In particular, the associated Dirichlet form of any reflecting Brownian motion on a Jordan domain always has proper regular subspaces.
\end{proposition}

In the next section, we will prove that the existence of the proper regular subspaces of absorbing or reflecting Brownian motions will {be held} not only for these two dimensional cases but for that on arbitrary domains of $\mathbf{R}^d$.

\subsection{Absorbing and reflecting Brownian motions on arbitrary domains}\label{ARBAD}
Let $\Gamma$ be an arbitrary domain of $\mathbf{R}^d$. Set
\begin{equation}
    \begin{split}
        &H^1(\Gamma):=\{u\in L^2(\Gamma): \int_\Gamma|\nabla u(x)|^2dx<\infty\},\\
        &{\mathbf{D}_\Gamma(u,v)}:=\int_{\Gamma} \nabla u(x)\cdot \nabla v(x)dx,\quad \forall u,v\in H^1(\Gamma),
    \end{split}
\end{equation}
and denote the closure of $C_\mathrm{c}^\infty(\Gamma)$ in $H^1(\Gamma)$ by $H^1_0(\Gamma)$. {By regarding $H_0^1(\Gamma)\subset H_0^1(\mathbf{R}^d)=H^1(\mathbf{R}^d)$, we can write $\mathbf{D}_\Gamma(u,v)=\mathbf{D}(u,v)$ for $u,v\in H^1_0(\Gamma)$.}  It is well known that $(\frac{1}{2}\mathbf{D},H^1_0(\Gamma))$ is a regular Dirichlet form on $L^2(\Gamma)$ with the core $C_\mathrm{c}^\infty(\Gamma)$ and its associated Markov process is just the absorbing Brownian motion $B^\Gamma$ on $\Gamma$.

\begin{proposition}\label{ABMAD}
    The associated Dirichlet form $(\frac{1}{2}\mathbf{D},H^1_0(\Gamma))$ on $L^2(\Gamma)$ of the absorbing Brownian motion $B^\Gamma$ on $\Gamma$ always possesses proper regular subspaces.
\end{proposition}
\begin{proof}
    Clearly, there exists a rectangle $G=I_1\times \cdots I_d\subset \Gamma$ where $I_i$ is an open interval of $\mathbf{R}$ for any $1\leq i\leq d$. Choose $s_1,\cdots s_d\in \mathbf{S}(\mathbf{R})$ such that for some $1\leq i\leq d$, the Lebesgue measure of the set
    \[
        E_{s_i}^{I_i}:=\{x\in I_i:s'_i(x)=0\}
    \]
    is positive. Define
    \begin{equation}\label{MCC}
        \mathcal{C}:=C_\mathrm{c}^\infty(s_1)\otimes \cdots \otimes C_\mathrm{c}^\infty(s_d)
    \end{equation}
    where $C_\mathrm{c}^\infty(s_i)=\{{\phi\circ s_i:} \phi\in C_\mathrm{c}^\infty(J_i)\text{ where }J_i=s_i(\mathbf{R})\}$. Clearly, $\mathcal{C}$ is closable in $(\frac{1}{2}\mathbf{D},H^1(\mathbf{R}^d))$\footnote{Note that $\mathbf{D}$ is the corresponding integral on $\mathbf{R}^d$.} and its closure $(\mathcal{E},\mathcal{F}):=(\mathcal{E}^{s_1,\cdots,s_d},\mathcal{F}^{s_1,\cdots, s_d})$ is the corresponding regular Dirichlet form of $X=(X^{s_1},\cdots,X^{s_d})$ which is defined in the notes before \text{Theorem \ref{PROP7}}. Since $|E_{s_i}^{I_i}|>0$ for some $i$, $X$ is different to $B$ and $(\mathcal{E},\mathcal{F})$ is a proper regular subspace of $(\frac{1}{2}\mathbf{D},H^1(\mathbf{R}^d))$.

    As usual, the part Dirichlet form $(\mathcal{E}_\Gamma,\mathcal{F}_\Gamma)$ of $(\mathcal{E},\mathcal{F})$ on open set $\Gamma$ can be written as:
    \begin{equation}\label{EFGM}
\begin{split}
        \mathcal{F}_\Gamma&=\{u\in\mathcal{F}:\tilde{u}=0\text{ q.e. on }\Gamma^c\}\\
        \mathcal{E}_\Gamma(u,v)&=\mathcal{E}(u,v),\quad  u,v\in \mathcal{F}_\Gamma.
    \end{split}
\end{equation}
    It is well known that $(\mathcal{E}_\Gamma,\mathcal{F}_\Gamma)$ is a regular Dirichlet form on $L^2(\Gamma)$ with the core
    \[
        \mathcal{C}_\Gamma:=\{u\in \mathcal{C}:\text{supp}[u]\subset \Gamma\}.
    \]
    On the other hand, $(\frac{1}{2}\mathbf{D},H^1_0(\Gamma))$ is just the part Dirichlet form of $(\frac{1}{2}\mathbf{D},H^1(\mathbf{R}^d))$ on $\Gamma$. Hence it is easy to check that
    \[
        (\mathcal{E}_\Gamma,\mathcal{F}_\Gamma)\prec (\frac{1}{2}\mathbf{D},H^1_0(\Gamma)).
    \]
    Note that the part Dirichlet form $(\mathcal{E}_G,\mathcal{F}_G)$ of $(\mathcal{E},\mathcal{F})$ on $G$ is a regular Dirichlet form with the core
    \[
        \mathcal{C}_G:=\{u\in \mathcal{C}:\text{supp}[u]\subset G\}.
    \]
    On the other hand,
    \[
        \mathcal{C}_G=C_\mathrm{c}^\infty(s_i|_{I_i})\otimes\cdots \otimes C_\mathrm{c}^\infty(s_d|_{I_d})
        \]
    where $C_\mathrm{c}^\infty(s_i|_{I_i})=\{\phi(s_i|_{I_i}): \phi\in C_\mathrm{c}^\infty(J'_i)\text{ where }J'_i=s_i(I_i)\}$. As $|E_{s_i}^{I_i}|>0$ for some $i$, $(\mathcal{E}_G,\mathcal{F}_G)$ is a proper regular subspace of $(\frac{1}{2}\mathbf{D},H^1_0(G))$.

    Now we are going to prove that $\mathcal{F}_\Gamma\varsubsetneqq H^1_0(\Gamma)$. If not, for any $f\in C_\mathrm{c}^\infty(G)\subset H^1_0(\Gamma)=\mathcal{F}_\Gamma$, there exists $\{u_n\}\subset \mathcal{C}_\Gamma$ such that $u_n\rightarrow f$ in the norm $||\cdot ||_{H^1(\Gamma)}$. Let $K=\text{supp}[f]\subset G$. Clearly $K$ is compact. Then there exists a $\varphi\in \mathcal{F}_G\cap C_\mathrm{c}(G)$ such that $\varphi|_K\equiv 1$. Let $v_n:=u_n\cdot \varphi\in \mathcal{F}_G$, it is easy to prove that $||v_n-f||_{H^1(G)}\rightarrow 0$ as $n\rightarrow \infty$. Hence $H^1_0(G)\subset \mathcal{F}_G$ which contradicts that $\mathcal{F}_G\subsetneqq H^1_0(G)$.
\end{proof}

Usually $(\frac{1}{2}{\mathbf{D}_\Gamma}, H^1(\Gamma))$ is a Dirichlet form on $L^2(\Gamma)$ but not regular. However under some conditions, for example, when the domain $\Gamma$ has \textbf{continuous boundary} in the following sense: any $x\in \partial D$ has a neighborhood $U$ such that 
\[
	\Gamma\cap U=\{(x_1,\cdots,x_d):x_d>F(x_1,\cdots, x_{d-1})\}\cap U
\] 
in some coordinate $(x_1,\cdots, x_d)$ and with a continuous function $F$, then 
\[
	C_\mathrm{c}^\infty(\bar{\Gamma}):=\{u|_{\bar{\Gamma}}:u\in C_\mathrm{c}^\infty(\mathbf{R}^d)\}
\]
is dense in $H^1(\Gamma)$ where $u|_{\bar{\Gamma}} $ is the restriction of $u$ on $\bar{\Gamma}$. Hence $(\frac{1}{2}{\mathbf{D}_\Gamma},H^1(\Gamma))$ is a regular Dirichlet form on $L^2(\bar{\Gamma})$, where $L^2(\bar{\Gamma}):=L^2(\bar{\Gamma},1_{\bar{\Gamma}}(x)\cdot dx)$. Its corresponding Markov process is just the reflecting Browninan motion $B^{\bar{\Gamma}}$ on $\bar{\Gamma}$.

\begin{proposition}
	Let $\Gamma$ be a domain of $\mathbf{R}^d$ with continuous boundary. Then the associated Dirichlet form $(\frac{1}{2}{\mathbf{D}_\Gamma},H^1(\Gamma))$ on $L^2(\bar{\Gamma})$ of the reflecting Brownian motion $B^{\bar{\Gamma}}$  on $\bar{\Gamma}$ always possesses proper regular subspaces.
\end{proposition}
\begin{proof}
	Similarly as in the proof of \text{Proposition \ref{ABMAD}}, choose an open rectangle $G=I_1\times \cdots I_d\subset \Gamma$ and $s_1,\cdots s_d\in \mathbf{S}(\mathbf{R})$ such that for some $1\leq i\leq d$, the Lebesgue measure of the set $ E_{s_i}^{I_i}$ is positive. We also define $\mathcal{C}$ as (\ref{MCC}). Define:
\[
	\mathcal{C}(\bar{\Gamma}):=\{u|_{\bar{\Gamma}}:u\in \mathcal{C}\}.
\]
Since $\mathcal{C}$ is dense in $C_\mathrm{c}(\mathbf{R}^d)$, it also follows that $\mathcal{C}(\bar{\Gamma})$ is dense in $C_\mathrm{c}(\bar{\Gamma}):=\{u|_{\bar{\Gamma}}:u\in C_\mathrm{c}(\mathbf{R}^d)\}$. It is easy to check that $(\frac{1}{2}{\mathbf{D}_\Gamma},\mathcal{C}(\bar{\Gamma}))$ is closable on $L^2(\bar{\Gamma})$. {Its closure denoted by} $(\mathcal{E}_{\bar{\Gamma}},\mathcal{F}_{\bar{\Gamma}})$, i.e. $\mathcal{F}_{\bar{\Gamma}}=\overline{\mathcal{C}(\bar{\Gamma})}^{H^1(\Gamma)}, \mathcal{E}_{\bar{\Gamma}}(u,v)=\frac{1}{2}{\mathbf{D}_\Gamma}(u,v)$ for any $u,v\in \mathcal{F}_{\bar{\Gamma}}$,  is a regular Dirichlet form on $L^2(\bar{\Gamma})$. Clearly,
\begin{equation}\label{EFBGM}
	(\mathcal{E}_{\bar{\Gamma}},\mathcal{F}_{\bar{\Gamma}})\prec (\frac{1}{2}{\mathbf{D}_\Gamma},H^1(\Gamma)).
\end{equation}
Note that $\mathcal{C}(\bar{G}):=\{u|_{\bar{G}}:u\in \mathcal{C}\}=C_\mathrm{c}^\infty(s_i|_{\bar{I}_i})\otimes\cdots \otimes C_\mathrm{c}^\infty(s_d|_{\bar{I}_d})$
    where $C_\mathrm{c}^\infty(s_i|_{\bar{I}_i})=\{\phi(s_i|_{\bar{I}_i}): \phi\in C_\mathrm{c}^\infty(\bar{J}'_i)\text{ where }\bar{J}'_i=s_i(\bar{I}_i)\}$. Clearly $G$ has the continuous boundary. It follows from $|E_{s_i}^{I_i}|>0$ for some $i$ and \text{Corollary \ref{COR5}} that the closure $(\mathcal{E}_{\bar{G}},\mathcal{F}_{\bar{G}})$ of $\mathcal{C}(\bar{G})$ is a proper regular subspace of $(\frac{1}{2}{\mathbf{D}_G},H^1(G))$ on $L^2(\bar{G})$.

Now we are going to prove $\mathcal{F}_{\bar{\Gamma}}\subsetneqq H^1(\Gamma)$. If $\mathcal{F}_{\bar{\Gamma}}=H^1(\Gamma)$, for any $f\in C_\mathrm{c}^\infty(\mathbf{R}^d)$, $f|_{\bar{\Gamma}}\in H^1(\Gamma)$, there exists $\{u_n\}\subset \mathcal{C}(\bar{\Gamma})$ such that $u_n\rightarrow f|_{\bar{\Gamma}}$ in $H^1(\Gamma)$. Hence it also holds that $u_n|_{\bar{G}}\rightarrow f|_{\bar{G}}$ in $H^1(G)$. It follows from $u_n|_{\bar{G}}\in \mathcal{C}(\bar{G})\subset \mathcal{F}_{\bar{G}}$ that $C_\mathrm{c}^\infty(\bar{G})\subset   \mathcal{F}_{\bar{G}}$ and hence $\mathcal{F}_{\bar{G}}=H^1(G)$ which contradicts that $(\mathcal{E}_{\bar{G}},\mathcal{F}_{\bar{G}})$ is a proper regular subspace of $(\frac{1}{2}{\mathbf{D}_G}, H^1(G))$ on $L^2(\bar{G})$.
\end{proof}


\subsection{Symmetric uniformly elliptic diffusions}

Let $\Gamma$ be an arbitrary domain of $\mathbf{R}^d$. In this section we pay special attention to the Markov symmetric form on $L^2(\Gamma)$:
\begin{equation}
	\begin{split}
        \mathcal{D}[\mathcal{E}]&:=C_\mathrm{c}^\infty (\Gamma),  \\
		\mathcal{E}(u,v)&:=\sum_{i,j=1}^d\int_\Gamma a_{ij}(x)\frac{\partial u(x)}{\partial x_i}\frac{\partial v(x)}{\partial x_j}dx,
	\end{split}
\end{equation}
where $a_{ij}$ are locally bounded Borel  functions on $\Gamma$ for any $1\leq i,j \leq d$ such that for any $\xi\in \mathbf{R}^d, x\in \Gamma$
\begin{equation}\label{AC1}
	\sum_{i,j=1}^d a_{ij}(x)\xi_i \xi_j\geq 0,\quad a_{ij}(x)=a_{ji}(x),\;1\leq i,j\leq d.
\end{equation}
Moreover, $(a_{ij})_{d\times d}$ is said to satisfy {the \emph{uniform ellipticity} if there exists $\delta>0$ such that}
\begin{equation}\label{AC2}
	\sum_{i,j=1}^d a_{ij}(x)\xi_i\xi_j \geq \delta |\xi|^2,\quad  \xi\in \mathbf{R}^d, x\in \Gamma.
\end{equation}
If $(a_{ij})$ satisfies the uniform ellipticity, then $(\mathcal{E},\mathcal{D}[\mathcal{E}])$ is closable. The closure of $\mathcal{D}[\mathcal{E}]$ is denoted by $\mathcal{F}$. Clearly $(\mathcal{E},\mathcal{F})$ is a regular Dirichlet form on $L^2(\Gamma)$. For more details, see \S3.1 of  \cite{FU}.

\begin{proposition}\label{UED}
	Let $(\mathcal{E},\mathcal{F})$ be above where $(a_{ij})$ are locally bounded Borel functions on $\Gamma$ and satisfy (\ref{AC1}) and (\ref{AC2}).  For any proper $(\frac{1}{2}\mathbf{D},\mathcal{G})\prec (\frac{1}{2}\mathbf{D}, H^1_0(\Gamma))$ with a core $\mathcal{C}$, it holds that $(\mathcal{E},\mathcal{C})$ is closable on $L^2(\Gamma)$ and its closure $(\mathcal{E},\mathcal{F}^{\mathcal{C}})$ is a proper regular subspace of $(\mathcal{E},\mathcal{F})$. In particular,  $(\mathcal{E},\mathcal{F})$ always possesses proper regular subspaces. 
\end{proposition}
\begin{proof}
	We need only to  prove $\mathcal{C}\subset \mathcal{F}$. In fact, if we have proved $\mathcal{C}\subset \mathcal{F}$, then clearly $(\mathcal{E},\mathcal{C})$ is closable and its closure $(\mathcal{E},\mathcal{F}^{\mathcal{C}})$ is a regular Dirichlet form with $\mathcal{F}^{\mathcal{C}}\subset \mathcal{F}$.  If $(\mathcal{E},\mathcal{F}^{\mathcal{C}})$ is not a proper subspace, in other words, $\mathcal{F}^{\mathcal{C}}=\mathcal{F}$, then for any $u\in C_\mathrm{c}^\infty(\Gamma)$ there exists $\{u_n\}\subset \mathcal{C}$ such that $u_n\rightarrow u$ in the norm $||\cdot ||_{\mathcal{E}_1}$.  It follows (\ref{AC2}) that $u_n\rightarrow u$ also in $H^1(\Gamma)$, hence $C_\mathrm{c}^\infty(\Gamma)\subset \mathcal{G}$ which contradits that  $(\frac{1}{2}\mathbf{D},\mathcal{G})$ is a proper regular subspace of $(\frac{1}{2}\mathbf{D}, H^1_0(\Gamma))$.

We are now ready to prove $\mathcal{C}\subset \mathcal{F}$.  For any $u\in \mathcal{C}$, let $K:=\text{supp}[u]$. Clearly $K\subset \Gamma$ is compact. Choose a $\phi\in C_\mathrm{c}^\infty(\mathbf{R}^d)$ such that $\phi$ supports on $\{x:|x|\leq 1\}$ and let $\phi_n(x):=n^d\phi(nx)$. Define $u_n:=\phi_n*u\in C_\mathrm{c}^\infty(\mathbf{R}^d)$. It is easy to check that there exists another compact set $\tilde{K}\supset K$ and $\tilde{K}\subset \Gamma$ such that $\text{supp}[u_n]\subset \tilde{K}$, and hence $u_n\in C_\mathrm{c}^\infty(\Gamma)$, for enough large $n$. As $(a_{ij})$ is locally bounded, there exists a constant $M>0$ such that $|a_{ij}(x)|<M$ for any $1\leq i,j\leq d$ and $x\in \tilde{K}$. On the other hand, $u_n\rightarrow u$ in $H^1(\Gamma)$. Hence we have:
\[\begin{split}
	\mathcal{E}(u_n-u,u_n-u)&=\int_{\Gamma}\sum_{i,j=1}^d a_{ij}(x)\frac{\partial (u_n-u)}{\partial x_i}\frac{\partial (u_n-u)}{\partial x_j} dx \\	
		&=\int_{\tilde{K}}\sum_{i,j=1}^d a_{ij}(x)\frac{\partial (u_n-u)}{\partial x_i}\frac{\partial (u_n-u)}{\partial x_j} dx \\
		&\leq M\cdot \int_{\tilde{K}}\sum_{i,j=1}^d \bigg|\frac{\partial (u_n-u)}{\partial x_i}\frac{\partial (u_n-u)}{\partial x_j}\bigg| dx \\
		&\leq {d\cdot M}\cdot\sum_{i=1}^d  \int_{\tilde{K}}\bigg|\frac{\partial (u_n-u)}{\partial x_i}\bigg|^2dx \\
		&\rightarrow 0
\end{split}\]
as $n\rightarrow \infty$. It follows that $u\in \mathcal{F}$ which complete the proof.
\end{proof}

\section*{Acknowledgement}

The authors would like to thank Prof. Zhen-Qing Chen for stimulating discussions and many helpful suggestions.

\end{document}